\documentclass[letterpaper]{amsart}

\usepackage[utf8]{inputenc}
\usepackage{lmodern}
\usepackage[T1]{fontenc}
\usepackage{amssymb}
\usepackage{enumitem}
\usepackage{mathtools}

\usepackage[pdfusetitle]{hyperref}
\usepackage{cleveref}

\allowdisplaybreaks[4]

\mathtoolsset{mathic}
\DeclareMathAlphabet\mathbfit{OML}{cmm}{b}{it}
\let\intertext\shortintertext
\let\setminus\smallsetminus

\newlist{enumarabic}{enumerate}{1}
\setlist[enumarabic]{font=\normalfont,label=(\arabic*),leftmargin=0.3in}
\newlist{enumroman}{enumerate}{1}
\setlist[enumroman]{font=\normalfont,label=(\roman*),leftmargin=0.3in}

\theoremstyle{plain}
\newtheorem{theorem}{Theorem}[section]
\newtheorem{proposition}[theorem]{Proposition}
\newtheorem{lemma}[theorem]{Lemma}
\newtheorem{corollary}[theorem]{Corollary}

\theoremstyle{definition}

\newtheorem{remark}[theorem]{Remark}
\newtheorem{example}[theorem]{Example}

\theoremstyle{remark}
\newtheorem*{acknowledgements}{Acknowledgements}

\numberwithin{equation}{section}

\let\newterm\emph

\def\arxiv#1{\href{http://arxiv.org/abs/#1}{\texttt{arXiv:#1}}}
\def\doi#1{\href{http://dx.doi.org/#1}{doi:#1}}

\def\cf{\emph{cf.}}

\let\epsilon\varepsilon
\let\emptyset\varnothing

\def\setone#1{\underline{#1}}
\def\setzero#1{[#1]}

\def\N{\mathbb N}
\def\Z{\mathbb Z}

\DeclareMathOperator{\Hom}{Hom}

\def\kk{\Bbbk}

\def\deg#1{|#1|}

\def\bigpair#1{\bigl\langle#1\bigr\rangle}

\def\Kan{G}

\def\JJ{b}

\def\EE{\mathbf{E}}
\def\EEt{\tilde{\EE}}
\def\EEE{\mathfrak{E}}
\def\FFF{\mathfrak{F}}

\def\OM{\boldsymbol{\Omega}}
\def\al{\boldsymbol{\alpha}}
\def\Ess{\mathcal{E}}
\def\Fss{\mathcal{F}}

\def\susp{\mathbf{s}}
\def\desusp{\susp^{-1}}
\def\Deltabar{\bar\Delta}

\def\muomc{\mu_{\OM C}}

\def\tpartial{\tilde\partial}

\def\AW{AW}
\def\AWu#1{\AW_{\mkern -1mu #1}}

\DeclareMathOperator{\pos}{pos}
\DeclareMathOperator{\perm}{perm}
\DeclareMathOperator{\Desusp}{des}

\DeclareMathOperator{\id}{id}
\DeclareMathOperator{\im}{im}

\DeclareMathOperator{\Sz}{Sz}
\DeclareMathOperator{\hatSz}{\widehat{S}z}
\def\DD#1#2{D_{#2,#1}}
\def\DDd#1#2{\DD{#1}{#2}^{\,\prime}}
\def\ii{\mathbfit{i}}
\def\jj{\mathbfit{j}}
\def\hjj{\mathbfit{k}}
\def\pp{\mathbfit{p}}

\DeclareMathOperator{\Shuff}{Shuff}

\def\swee#1#2{#1_{(#2)}}

\def\GG{G}
\def\OMex{\boldsymbol{\tilde\Omega}}
\def\Tri{\mathcal T}
\def\TT{\mathbb T}

\def\Cex{\tilde C}
\def\tex{\tilde t}

\def\rr#1{\stackrel{#1}{\rule[0.55ex]{1.6em}{0.5pt}}}
\def\rrbf#1{\rr{\boldsymbol{#1}}}

\def\CobarEl#1{\langle#1\rangle}
\def\bigCobarEl#1{\bigl\langle\,#1\,\bigr\rangle}

\def\aa{\mathfrak a}
\def\FF{\mathcal F}
\DeclareMathOperator{\gr}{gr}

\def\bigdeg#1{\bigl|#1\bigr|}

\let\shuffle\nabla
\def\iter#1#2{#1^{[#2]}}

\def\timesunder#1{\mathbin{\mathchoice
  {\mathop\times\limits_{\mkern-5mu #1\mkern-5mu}}%
  {\times_{#1}}{\times_{#1}}{\times_{#1}}}}

\def\otimesunder#1{\mathbin{\mathchoice
  {\mathop\otimes\limits_{\mkern-5mu #1\mkern-5mu}}%
  {\otimes_{#1}}{\otimes_{#1}}{\otimes_{#1}}}}

\begin{document}

\title{Szczarba's twisting cochain is comultiplicative}
\author{Matthias Franz}
\thanks{The author was supported by an NSERC Discovery Grant.}
\address{Department of Mathematics, University of Western Ontario, London, Ont.\ N6A\;5B7, Canada}
\email{mfranz@uwo.ca}


\begin{abstract}
  We prove that Szczarba's twisting cochain is comultiplicative.
  In particular, the induced map from the cobar construction~\(\OM\,C(X)\)
  of the chains on a \(1\)-reduced simplicial set~\(X\) to~\(C(\Kan X)\),
  the chains on the Kan loop group of~\(X\), is a quasi-isomorphism of dg~bialgebras.
  We also show that Szczarba's twisted shuffle map is a dgc map
  connecting a twisted Cartesian product with the associated twisted tensor product.
  This gives a natural dgc model for fibre bundles.
  We apply our results to finite covering spaces and to the Serre spectral sequence.
\end{abstract}

\maketitle

\section{Introduction}

Let \(X\) be a simplicial set and \(G\) a simplicial group. Given a twisting function
\begin{equation}
  \tau\colon X_{>0} \to G,
\end{equation}
Szczarba~\cite{Szczarba:1961} has constructed an explicit twisting cochain
\begin{equation}
  \label{eq:twc-intro}
  t\colon C(X) \to C(G).
\end{equation}
In~\cite[Thm.~6.2]{Franz:szczarba1} we showed that it agrees with the twisting cochain obtained by Shih~\cite[\S II.1]{Shih:1962}
using homological perturbation theory if one uses a slightly modified version of the Eilenberg--Mac\,Lane homotopy.

In this note we consider the associated map of differential graded algebras (dgas)
\begin{equation}
  \label{eq:OMCX-CG}
  \OM \,C(X) \to C(G)
\end{equation}
where \(\OM\,C(X)\) is the reduced cobar construction of the differential graded coalgebra (dgc)~\(C(X)\).

The diagonal of~\(C(G)\) is compatible with the multiplication, meaning that \(C(G)\) is actually a dg~bialgebra.
By work of Baues~\cite{Baues:1981} and Gerstenhaber--Voronov~\cite[Cor.~6]{GerstenhaberVoronov:1995}, the same holds true for~\(\OM \,C(X)\). Here the diagonal
can be expressed via the homotopy Gerstenhaber structure of~\(C(X)\), that is, in terms of certain cooperations
\begin{equation}
  E^{k}\colon C(X) \to C(X)^{\otimes k}\otimes C(X)
\end{equation}
with~\(k\ge0\), see \Cref{sec:hgc}.

The question arises as to whether the dga map~\eqref{eq:OMCX-CG} is comultiplicative, meaning compatible with the diagonals.
Hess--Parent--Scott--Tonks~\cite[Thm.~4.4]{HessEtAl:2006} showed that for \(1\)-reduced~\(X\) this is true up to homotopy in a strong sense,
and they observed that it holds on the nose up to degree~\(3\).
In the case where \(X\) is a simplicial suspension the comultiplicativity was established by Hess--Parent--Scott~\cite[Thm.~4.11]{HessEtAl:2007}.
Our main result says that it is true in general.

\begin{theorem}
  \label{thm:main}
  Let \(X\ne\emptyset\) be a simplicial set, and let \(G\) and~\(\tau\) be as above.
  The dga map~\(\OM \,C(X) \to C(G)\) induced by Szczarba's twisting cochain~\(t\) is comultiplicative.
\end{theorem}

This applies in particular to the canonical twisting cochain~\(\tau_{X}\colon X_{>0}\to\Kan X\) of a \(1\)-reduced simplicial set
where \(\Kan X\) denotes the Kan loop group of~\(X\). This gives the following.

\begin{corollary}
  \label{thm:cobar-quiso}
  For \(1\)-reduced~\(X\),
  the map~\(\OM\,C(X) \to C(\Kan X)\) induced by Szczarba's twisting cochain~\(t\) is a quasi-iso\-mor\-phism of dg~bialgebras.
\end{corollary}

Using Hess--Tonks' extended cobar construction~\(\OMex\,C(X)\), we generalize this to reduced simplicial sets in \Cref{thm:excobar-comult}.
After the prepublication of this article, Medina-Mardones and Rivera
showed that \(\OMex\,C(X)\) and~\(C(\Kan X)\) are quasi-iso\-mor\-phic as \(E_{\infty}\)-coalgebras \cite[Thm.~2]{MMRivera:2021}.
This quasi-isomorphism involves a zigzag, however, not the extension~\(\OMex\,C(X) \to C(\Kan X)\) of the map above.

Given a left \(G\)-space~\(F\), one can consider
the twisted Cartesian product~\(X\times_{\tau}F\) as well as
the twisted tensor product~\(C(X)\otimes_{t}C(F)\).
Dualizing a construction due to Kadeishvili--Saneblidze~\cite{KadeishviliSaneblidze:2005},
we turn the latter into a dgc. Szczarba also defined a twisted shuffle map
\begin{equation}
  \label{eq:psi-intro}
  \psi\colon C(X)\otimes_{t}C(F) \to C(X\times_{\tau}F)
\end{equation}
and proved that it is a quasi-isomorphism of complexes \cite{Szczarba:1961}.
In~\cite[Prop.~7.1]{Franz:szczarba1} we showed that \(\psi\) is in fact a morphism of left \(C(X)\)-comodules,
and also of right \(C(G)\)-modules in the case~\(F=G\). We strengthen the first aspect as follows.

\begin{theorem}
  \label{thm:psi-dgc}
  Szczarba's twisted shuffle map~\(\psi\) is a quasi-isomorphism of dgcs.
  In particular, the twisted tensor product~\(C(X)\otimes_{t}C(F)\) is a dgc model for~\(X\times_{\tau}F\).
\end{theorem}

Using cubical chains, Kadeishvili--Saneblidze~\cite[Sec.~6]{KadeishviliSaneblidze:2005}
have previously obtained a dgc model for fibre bundles with simply connected base.\footnote{\label{fn:kadeishvili-saneblidze-pi1}%
  The assumption of simple connectedness is omitted in the statement of~\cite[Thm.~6.1]{KadeishviliSaneblidze:2005}, but used in the proof.
  That proof is actually problematic because it refers to a map of monoidal cubical sets~\(\operatorname{Sing}^{I}\Omega Y\to\operatorname{Sing}^{I}G\)
  whose existence is doubtful.}

Content and structure of this paper are as follows: We review background material in \Cref{sec:prelim}
and homotopy Gerstenhaber coalgebras in \Cref{sec:hgc}. After establishing a purely combinatorial result in \Cref{sec:bijection}
and discussing the Szczarba maps in \Cref{sec:szczarba} we prove \Cref{thm:main} in \Cref{sec:proof-main}.
The generalization of \Cref{thm:cobar-quiso} mentioned above appears in \Cref{sec:cobar}. In \Cref{sec:twisted-tensor}
we explain how homotopy Gerstenhaber coalgebra structures give rise to dgc structures on twisted tensor products,
and in \Cref{sec:morph-dgc} we prove \Cref{thm:psi-dgc}.
In \Cref{sec:shih} we compare Szczarba's twisted tensor product with a similar one due to Shih~\cite{Shih:1962}.
In \Cref{sec:cover} we deduce from our results a dga model for finite covering spaces
as well as certain spectral sequences 
studied by Papadima--Suciu~\cite{PapadimaSuciu:2010} and Rüping--Stephan~\cite{RuepingStephan:2019}.
In a similar vein, we obtain the (co)mul\-ti\-plica\-tive structure of the (co)ho\-mo\-log\-i\-cal Serre spectral sequence in \Cref{sec:serre}.
In the first appendix we relate our diagonal on~\(\OM\,C(X)\) to the one defined by Baues~\cite{Baues:1981} for \(1\)-reduced~\(X\).
In the second we fill a gap in the literature by showing that Szczarba's twisting cochain~\eqref{eq:twc-intro} 
and twisted shuffle map~\eqref{eq:psi-intro} are in fact well-de\-fined on normalized chain complexes.

\begin{acknowledgements}
  The authors thanks the anonymous referee for helpful suggestions
  and Hanibal Medina-Mardones for explaining his work and pointers to the literature.
\end{acknowledgements}

\section{Preliminaries}
\label{sec:prelim}

\subsection{Generalities}

We write \(\setzero{n}=\{0,\dots,n\}\) and \(\setone{n}=\{1,\dots,n\}\) for~\(n\ge0\).
We work over a commutative ring~\(\kk\) with unit;
all tensor products and chain complexes are over~\(\kk\).
Unless specified otherwise, all chain complexes are homological.
The degree of an element~\(c\) of a graded module~\(C\) is denoted by~\(\deg{c}\).
We write \(1_{C}\) for the identity map of~\(C\) and
\begin{equation}
  T_{B,C}\colon B\otimes C\to C\otimes B,
  \qquad
  b\otimes c\mapsto (-1)^{\deg{b}\deg{c}}\, c\otimes b
\end{equation}
for the transposition of factors in a tensor product of graded modules.
The suspension and desuspension operators are denoted by~\(\susp\) and~\(\desusp\), respectively.
We systematically use the Koszul sign rule, compare~\cite[Secs.~2.2~\&~2.3]{Franz:gersten}.

For clarity, we sometimes write \(1_{A}\) for the unit of a dga~\(A\) and \(1_{C}\) for the unit of a coaugmented dgc~\(C\).
A dg~bialgebra is a chain complex~\(A\) that is both a dga and a dgc in such a way that each pair of structure maps
are morphisms with respect to the other structure.

We write \(C(X)\) for the normalized chains on a simplicial set~\(X\).
We also write \(\tpartial\) for the last face map, that is, \(\tpartial x=\partial_{n}x\) for~\(x\in X_{n}\) with~\(n\ge1\).

\subsection{The cobar construction}

Let \(C\) be a dgc with coaugmentation~\(\iota\colon\kk\hookrightarrow C\), so that \(C=\kk\oplus\bar C\) where \(\bar C=\ker\epsilon\).
The (reduced) cobar construction of~\(C\) is
\begin{equation}
  \OM\,C = \bigoplus_{k\ge0} \OM_{k}\,C
  \qquad\text{where}\qquad
  \OM_{k}\,C = (\desusp\bar C)^{\otimes k},
\end{equation}
compare~\cite[Sec.~II.3]{HusemollerMooreStasheff:1974} or~\cite[\S 0]{Baues:1981}.
We write elements of~\(\OM\,C\) in the form
\begin{equation}
  \CobarEl{c_{1}|\dots|c_{k}} = \desusp\,c_{1} \otimes \dots \otimes \desusp\,c_{k}
\end{equation}
with~\(c_{1}\),~\dots,~\(c_{k}\in\bar C\).
The cobar construction is an augmented dga with concatenation as product and unit~\(1=\CobarEl{}\in\OM_{0}\,C=\kk\).
The differential and augmentation are determined by
\begin{equation}
  d\,\CobarEl{c} = -\CobarEl{d c} + (\desusp\otimes\desusp)\,\Deltabar\,c
  \qquad\text{and}\qquad
  \epsilon(\CobarEl{c}) = 0
\end{equation}
for~\(\CobarEl{c}\in\OM_{1}\,C\), where
\begin{equation}
  \label{eq:def-Deltabar}
  \Deltabar\,c = \Delta\,c - c\otimes 1 - 1\otimes c \in \bar C \otimes \bar C
\end{equation}
is the reduced diagonal.

\subsection{Twisting cochains}

Let \(C\) be a coaugmented dgc and \(A\) an augmented dga.
Recall that the complex~\(\Hom(C,A)\) is an augmented dga via
\begin{gather}
  d(f) = d_{A}\,f - (-1)^{\deg{f}}\,f\,d_{C},
  \qquad
  1_{\Hom(C,A)} = \iota_{A}\,\epsilon_{C}
  \\
  f\cup g = \mu_{A}\,(f\otimes g)\,\Delta_{C},
  \qquad
  \epsilon(f) = \epsilon_{A}\,f\,\iota_{C}(1)
\end{gather}
for~\(f\),~\(g\in\Hom(C,A)\). Here \(\iota_{A}\colon\kk\to A\) is the unit map, \(\iota_{C}\) is the coaugmentation of~\(C\),
and \(\epsilon_{C}\) and~\(\epsilon_{A}\) are the augmentations of~\(C\) and~\(A\), respectively.

A twisting cochain is a map~\(t\in\Hom(C,A)\)
of degree~\(-1\) (in the homological setting) such that
\begin{equation}
  \label{eq:def-tw}
  t\,\iota_{C} = 0,
  \qquad
  \epsilon_{A}\,t = 0,
  \qquad
  d(t) = t\cup t.
\end{equation}
It canonically induces the morphism of dgas
\begin{equation}
  \OM\,C \to A,
  \quad
  \CobarEl{c_{1}|\dots|c_{k}} \mapsto t(c_{1}) \cdots t(c_{k}).
\end{equation}
For example, the canonical twisting cochain
\begin{equation}
  t_{C}\colon C\to\OM\,C,
  \qquad
  c \mapsto \CobarEl{\bar c} \in\OM_{1} C
\end{equation}
corresponds to the identity map on~\(\OM\,C\).
Here we have written \(\bar c=c-\iota\,\epsilon(c)\) for the component of~\(c\) in~\(\bar C\).

\subsection{The shuffle map}
\label{sec:shuffle}

We recall the definition of the shuffle map for an arbitrary number of factors.
Given \(k\ge1\)~non-neg\-a\-tive integers~\(q_{1}\),~\dots,~\(q_{k}\) with sum~\(q\), a \((q_{1},\dots,q_{k})\)-shuffle
is a partition~\(\al=(\alpha_{1},\dots,\alpha_{k})\) of the set~\(\setzero{q-1}\).
We write \((-1)^{(\al)}=(-1)^{(\alpha_{1},\dots,\alpha_{k})}\) for its signature
and \(\Shuff(q_{1},\dots,q_{k})\) for the set of all such shuffles.
Observe that for~\(k=1\) there is only one \((q)\)-shuffle.

For simplicial sets~\(X_{1}\),~\dots,~\(X_{k}\) the shuffle map is given by
\begin{align}
  \label{eq:def-shuffle}
  \shuffle_{X_{1},\dots,X_{k}}\colon
  C_{q_{1}}(X_{1}) \otimes \dots \otimes C_{q_{k}}(X_{k}) &\to C_{q}(X_{1}\times \dots \times X_{k}), \\
  \notag x_{1}\otimes \dots \otimes x_{k} &\mapsto \sum_{\al} (-1)^{(\al)}\,(s_{\bar\alpha_{1}}x_{1},\dots,s_{\bar\alpha_{k}}x_{k})
\end{align}
where the sum is over all~\(\al\in\Shuff(q_{1},\dots,q_{k})\), and
\(\bar\alpha_{s}=\setzero{q-1}\setminus\alpha_{s}\) for~\(1\le s\le k\).

Using the shuffle map, one turns the chain complex of a simplicial group~\(G\) into a dga.
For~\(m\ge0\), the \(m\)-fold iterated multiplication is given by
\begin{equation}
  C(G)^{\otimes m} \xrightarrow{\shuffle_{G,\dots,G}} C(G\times\dots\times G) \xrightarrow{\iter{\mu}{m}_{*}} C(G)
\end{equation}
where \(\iter{\mu}{m}\colon G\times\dots\times G\to G\) is the \(m\)-fold product map.
This gives the identity map of~\(C(G)\) for~\(m=1\) and the unit map~\(\kk\hookrightarrow C(G)\) for~\(m=0\).

\subsection{Twisted Cartesian products}

Twisted Cartesian products are simplicial versions of fibre bundles, compare~\cite[Sec.~18]{May:1968} or~\cite[Sec.~1]{Szczarba:1961}.
More precisely, let \(X\) and~\(F\) be simplicial sets, and assume that the simplicial group~\(G\)
acts on~\(F\) from the left. The twisted Cartesian product~\(X\times_{\tau}F\)
differs from the usual Cartesian product~\(X\times F\) only by the zeroeth face map, which is
\begin{equation}
  \partial_{0}\,(x,y) = \bigl(\partial_{0}\,x,\tau(x)\,\partial_{0}\,y\bigr).
\end{equation}
The twisting function
\begin{equation}
  \tau\colon X_{>0} \to G
\end{equation}
is of degree~\(-1\) and for any~\(x\in X\) of dimension~\(n>0\) satisfies
\begin{align}
  \partial_{0}\,\tau(x) &= \tau(\partial_{0}\,x)^{-1}\,\tau(\partial_{1}\,x), \\
  \partial_{k}\,\tau(x) &= \tau(\partial_{k+1}\,x) \qquad \text{for~\(0<k<n\),} \\
  \label{eq:tau-sk}
   s_{k}\,\tau(x) &= \tau(s_{k+1}\,x) \qquad \text{for~\(0\le k<n\),} \\
\shortintertext{and for any~\(x\in X\) of dimension~\(n\ge0\) also}
  \label{eq:tau-s0}
  \tau(s_{0}\,x) &= 1 \in G_{n},
\end{align}
see~\cite[eq.~(1.1)]{Szczarba:1961},~\cite[Def.~18.3]{May:1968} or~\cite[Sec.~1.3]{HessTonks:2006}.

\subsection{Interval cut operations}

Let \(k\),~\(l\ge0\), and
let \(u\colon\setone{k+l}\to\setone{k}\) be a surjection such that \(u(i)\ne u(i+1)\) for all~\(0\le i<k+l\).
Berger--Fresse~\cite[Sec.~2]{BergerFresse:2004} have associated to~\(u\) an \newterm{interval cut operation}
\begin{equation}
  \AWu{u}\colon C(X) \to C(X)^{\otimes k},
\end{equation}
natural in the simplicial set~\(X\). On an \(n\)-simplex~\(x\in X\), it is given by
\begin{equation}
  \label{eq:def-AWu}
  \AWu{u}\,x = \sum_{\pp} (-1)^{\pos(\pp)+\perm(\pp)}\, x^{\pp}_{1} \otimes \dots \otimes x^{\pp}_{k}.
\end{equation}
Here the sum runs over all decompositions~\(\pp=(0=p_{0},p_{1},\dots,p_{k+l}=n)\) of~\(\setzero{n}\) into \(k+l\)~intervals.
If we think of these intervals as being labelled via~\(u\), then
\begin{equation}
  \label{eq:def-xpp-s}
  x^{\pp}_{s} = x(p_{i_{1}-1},\dots,p_{i_{1}},p_{i_{2}-1},\dots,p_{i_{2}},\dots,p_{i_{m}-1},\dots,p_{i_{m}})
\end{equation}
where \(i_{1}\),~\dots,~\(i_{m}\) enumerate the intervals with label~\(s\).
We refer to~\cite[\S 2.2.4]{BergerFresse:2004} for the definitions of 
the position sign exponent~\(\pos(\pp)\) and the permutation sign exponent~\(\perm(\pp)\).

Whenever we talk about the \newterm{length} of an interval~\([p_{i-1},\dots,p_{i}]\) in this paper,
we always mean its naive length~\(p_{i}-p_{i-1}\), not the possibly different length defined in~\cite[\S 2.2.3]{BergerFresse:2004}
to compute the position and permutation sign exponents.

\section{Homotopy Gerstenhaber coalgebras}
\label{sec:hgc}

Homotopy Gerstenhaber coalgebras (hgcs) are defined such that their duals are homotopy Gerstenhaber algebras (hgas),
see \Cref{rem:hga} below and also~\cite[p.~223]{KadeishviliSaneblidze:2005}. More precisely,
an \newterm{hgc} is a coaugmented dgc~\(C\) together with a family of cooperations
\begin{equation}
  E^{k}\colon C \to C^{\otimes k}\otimes C
\end{equation}
for~\(k\ge0\) such that
\begin{gather}
  \label{eq:E0}
  E^{0}=1_{C}, \\
  \label{eq:Ek-coaugm}
  \im E^{k} \subset \bar C^{\otimes k} \otimes \bar C  \qquad\text{for~\(k>0\)}, \\
  \label{eq:Ek-0}
  E^{k}(c) = 0
  \qquad
  \text{for~\(\deg{c}<k\).}
\end{gather}
Recall that \(\bar C=\ker\epsilon\) is the augmentation ideal
and \(\bar c=c-\iota\,\epsilon(c)\) the component of~\(c\) in~\(\bar C\).
There are further conditions on the maps~\(E^{k}\). Defining
\begin{equation}
  \EE^{k}\colon C \to (\desusp\,\bar C)^{\otimes k}\otimes \desusp\,\bar C
  = \OM_{k} C\otimes \OM_{1} C \subset \OM\,C\otimes \OM\,C
\end{equation}
for~\(k\ge0\) via
\begin{equation}
  \susp^{\otimes(k+1)}\,\EE^{k}(c) = E^{k}(\bar c),
\end{equation}
the assignment
\begin{align}
  \label{eq:def-EE}
  \EE\colon C &\to \OM\,C\otimes \OM\,C, \\
  \notag c &\mapsto \CobarEl{\bar c} \otimes 1 + \sum_{k=0}^{\infty}\EE^{k}(c)
  = \CobarEl{\bar c} \otimes 1 + 1 \otimes \CobarEl{\bar c} + \sum_{k=1}^{\infty}\EE^{k}(c)
\end{align}
is well-defined by~\eqref{eq:Ek-0}.
We require \(\EE\) to be a twisting cochain and the associated dga map
\begin{equation}
  \label{eq:hgc-diag-cobar}
  \Delta\colon \OM\,C \to \OM\,C\otimes \OM\,C,
  \qquad
  \bigCobarEl{ c_{1} \bigm| \dots \bigm| c_{k} } \mapsto \EE(c_{1})\cdots \EE(c_{k})
\end{equation}
to be coassociative, so that \(\OM\,C\) becomes a dg~bialgebra.

It will be convenient to rephrase these conditions in terms of the function
\begin{align}
  \label{eq:def-EEE}
  \EEE \colon C &\to \OM\,C \otimes C, \\
  \notag c &\mapsto (1\otimes p_{C})\,\EE(c) + 1 \otimes \iota\,\epsilon(c) =  1\otimes c + (1\otimes p_{C})\sum_{k=1}^{\infty}\EE^{k}(c)
\end{align}
of degree~\(0\) where
\begin{equation}
  \label{eq:dec-pC}
  p_{C}\colon \OM\,C \longrightarrow \OM_{1} C = \desusp \bar C \stackrel{\susp}{\longrightarrow} \bar C \hookrightarrow C
\end{equation}
is the composition of the canonical projection, the suspension map and the canonical inclusion.
Like the suspension map, \(p_{C}\) is of degree~\(1\).

\goodbreak

\begin{lemma}
  \label{thm:conditions-EEE}
  Let \(\EE\) and~\(\EEE\) be as in~\eqref{eq:def-EE} and~\eqref{eq:def-EEE}.
  \begin{enumroman}
  \item That \(\EE\) is a twisting cochain is equivalent to the two identities
    \begin{gather*}
      d(\EEE) = (\muomc \otimes 1_{C})(t_{C} \otimes \EEE)\,\Delta_{C}
      - (\muomc \otimes 1_{C})\,(1_{\OM C}\otimes T_{C,\OM C})\,(\EEE \otimes t_{C})\,\Delta_{C}, \\
      (1_{\OM C} \otimes \Delta_{C})\,\EEE = (\muomc \otimes 1_{C} \otimes 1_{C})\,(1_{\OM C} \otimes T_{C,\OM C} \otimes 1_{C})\,(\EEE \otimes \EEE)\,\Delta_{C} .
    \end{gather*}
  \item Assume that \(\EE\) is a twisting cochain.
    The coassociativity of the diagonal~\eqref{eq:hgc-diag-cobar} then is equivalent to the formula
    \begin{equation*}
      (\Delta_{\OM C} \otimes 1_{C})\,\EEE = (1_{\OM C} \otimes \EEE)\,\EEE.
    \end{equation*}
  \end{enumroman}
\end{lemma}

\begin{proof}
  For the first part, we note that both sides of the twisting cochain condition \(d(\EE)=\EE\cup\EE\)
  only have components in~\(\OM\,C\otimes\OM_{l}\,C\) with~\(l\le2\). We project onto these components separately.
  The projections for~\(l=0\) are always equal. A direct calculation shows that the projections for~\(l=1\) and~\(l=2\)
  correspond to the two identities for~\(\EEE\) given above. It is helpful to distinguish the two cases~\(c=1\) and~\(c\in\bar C\),
  and in the second one to split up the diagonal as~\(\Delta\,c=c\otimes 1+1\otimes c+\Deltabar\,c\)
  where \(\Deltabar\) is the reduced diagonal~\eqref{eq:def-Deltabar}.
  For the first identity one also uses \(d(p_{C})=0\).

  The second claim follows similarly by projecting the coassociativity condition \((\Delta\otimes1)\,\Delta=(1\otimes\Delta)\,\Delta\)
  to~\(\OM\,C\otimes\OM\,C\otimes\OM_{1}\,C\).
\end{proof}

\begin{remark}
  \label{rem:hga}
  Let \(A=\Hom(C,\kk)\) be the augmented dga dual to the coaugmented dgc~\(C\).
  For~\(k\ge0\) define the transpose
  \begin{equation}
    E_{k}\colon A^{\otimes k}\otimes A \to A
  \end{equation}
  of the cooperation~\(E^{k}\) by
  \begin{equation}
    \bigpair{E_{k}(a_{1},\dots,a_{k};b),c} = (-1)^{k(\deg{a_{1}}+\dots+\deg{a_{k}}+\deg{b})}\,\bigpair{a_{1}\otimes\dots\otimes a_{k}\otimes b, E^{k}(c)}    
  \end{equation}
  for~\(c\in C\), compare~\cite[eq.~(4)]{Franz:gersten}.
  The operations~\(E_{k}\) then form an hga structure on~\(A\)
  that satisfies the analogues of the identities stated in~\cite[Sec.~6.1]{Franz:homog}.
  Note that in~\cite{Franz:homog} operations of the form~\(E_{k}(a;b_{1},\dots,b_{k})\) are used;
  see~\cite[Rem.~6.1]{Franz:homog} for their relation to the braces used by Gerstenhaber--Voronov~\cite{GerstenhaberVoronov:1995}.
  The explicit signs given there remain unchanged, except for an additional overall minus sign in the formula for~\(d(E_{k})\).
\end{remark}

Let \(t\colon C\to A\) be a twisting cochain, where \(C\) is an hgc and \(A\) a dg~bialgebra.
We say that \(t\) is \newterm{comultiplicative} if the induced dga map~\(\OM\,C \to A\)
is a morphism of dgcs and therefore of dg~bialgebras.
This definition is dual to Kadeishvili--Saneblidze's notion of a multiplicative twisting cochain \cite[Def.~7.2]{KadeishviliSaneblidze:2005}.
For example, the canonical twisting cochain~\(t_{C}\colon C\to\OM\,C\) is comultiplicative.

The normalized chains~\(C(X)\) on a simplicial set~\(X\ne\emptyset\) form an hgc in a natural way,
for any coaugmentation~\(\kk\hookrightarrow C(X)\) sending \(1\in\kk\) to some basepoint~\(x_{0}\in X\).
In terms of interval cut operations, the structure maps are given by
\begin{equation}
  \label{eq:def-E-CX}
  E^{k} = (-1)^{k}\,\AWu{e_{k}},
\end{equation}
that is,
\begin{equation}
  \label{eq:def-EE-CX}
  \EE^{k} = (-1)^{k(k-1)/2}\,(\desusp)^{\otimes(k+1)}\,\AWu{e_{k}}
\end{equation}
where
\begin{equation}
  e_{k} = (k+1,1,k+1,2,k+1,\dots,k+1,k,k+1).
\end{equation}
The sign difference in~\eqref{eq:def-EE-CX} compared to~\eqref{eq:def-E-CX}
stems from the fact that the (de)sus\-pen\-sion operators have degree~\(\pm1\), so that
\begin{equation}
  (\desusp)^{\otimes(k+1)}\,\susp^{\otimes(k+1)} = (-1)^{k(k+1)/2}
\end{equation}
because the sign changes each time an~\(\desusp\) is moved past an~\(\susp\) for a different tensor factor.
Note that \(\AWu{e_{0}}\) is the identity map as required by condition~\eqref{eq:E0}.
A look at the formula~\eqref{eq:def-xpp-s} moreover shows that
the intervals labelled \(1\),~\dots,~\(k\) in the surjection must have length at least~\(1\)
in order for the last factor~\(x^{\pp}_{k+1}\) of each term in the sum~\eqref{eq:def-AWu} for~\(\AWu{e_{k}}\)
to be non-degenerate, which confirms \eqref{eq:Ek-0} and also~\eqref{eq:Ek-coaugm}.

Explicitly, the induced diagonal on~\(\OM\,C(X)\) can be written as
\begin{equation}
  \label{eq:def-diag-cobar-CX}
  \Delta\,\CobarEl{x} = \EE(x) = \CobarEl{x} \otimes 1 +
  \sum_{k=0}^{n}\sum_{\pp}\,(-1)^{\epsilon(\pp)}\,\bigCobarEl{x^{\pp}_{1}\bigm|\dots\bigm|x^{\pp}_{k}} \otimes \bigCobarEl{x^{\pp}_{k+1}}
\end{equation}
for~\(x\in X_{n}\), where \(\pp\) runs through the cuts of~\(\setzero{n}\) prescribed by~\(e_{k}\). The sign exponent is given by
\begin{equation}
  \label{eq:def-epp}
  \epsilon(\pp) = \frac{k(k-1)}{2} + \Desusp(\pp) + \pos(\pp) + \perm(\pp)
\end{equation}
where
\begin{equation}
  \Desusp(\pp) = \sum_{s=1}^{k}\,(k+1-s)\,\deg{x^{\pp}_{s}} = \sum_{s=1}^{k}\,(k+1-s)(p_{2s}-p_{2s-1})
\end{equation}
is the sign exponent incurred by the desuspension operators in~\eqref{eq:def-EE}.
Note that for~\(n=0\) the formula~\eqref{eq:def-diag-cobar-CX} boils down
to~\(\Delta\,\CobarEl{x} = \CobarEl{x} \otimes 1 +  1 \otimes \CobarEl{x}\), so that
\(\CobarEl{x}\) is primitive for any \(0\)-simplex~\(x\in X_{0}\).\footnote{%
  Strictly speaking, we should write \(\CobarEl{\bar x}\), so that \(\CobarEl{\bar x}=0\) for the basepoint~\(x=x_{0}\).}

In  \Cref{sec:Baues} we show that for \(1\)-reduced~\(X\)
the diagonal~\eqref{eq:def-diag-cobar-CX} on~\(\OM\,C(X)\)
agrees with those defined by Baues \cite{Baues:1981}
and Hess--Parent--Scott--Tonks~\cite{HessEtAl:2006}.

\begin{lemma}
  \label{eq:epp-length0}
  Let \(k\),~\(n\ge 1\), and
  let \(\pp=(p_{0},\dots,p_{2k+1})\) be an interval cut of~\(\setzero{n}\) for the surjection~\(e_{k}\)
  such that all intervals with label~\(k+1\) have length~\(0\). Then
  \begin{equation*}
    \epsilon(\pp) \equiv \sum_{s=1}^{k}(s-1)(p_{2s}-p_{2s-1}-1) \pmod{2}.
  \end{equation*}
\end{lemma}

\begin{proof}
  Modulo~\(2\), we have
  \begin{align}
    \pos(\pp) &= p_{1} + p_{3} + \dots + p_{2k-1}, \\
    \perm(\pp) &= (p_{3}-p_{1}) + 2\cdot (p_{5}-p_{3}) + \dots + k\cdot(p_{2k+1}-p_{2k-1}) \\*
    \notag &\equiv p_{1} + p_{3} + \dots + p_{2k-1} + n\,k, \\
    \Desusp(\pp) &= \sum_{s=1}^{k}(k+1-s)(p_{2s}-p_{2s-1}) \\*
    \notag &\equiv n\,k + \sum_{s=1}^{k}(s-1)(p_{2s}-p_{2s-1}), \\
    \frac{k(k-1)}{2} &=\sum_{s=1}^{k}(s-1),
  \end{align}
  which gives the desired result.
\end{proof}

\begin{lemma}
  \label{thm:pp-refine}
  Let \(0\le m\le k\) and~\(n\ge1\), and let
  \begin{equation*}
    \pp: p_{0} \rr{k+1} p_{1} \rr{1} 
    \cdots \rr{m} p_{2m} \rrbf{k+1} p_{2m+1} \rr{m+1} \cdots
    \rr{k} p_{2k} \rr{k+1} p_{2k+1}
  \end{equation*}
  be an interval cut of~\(\setzero{n}\) corresponding to the surjection~\(e_{k}\).
  Assume that the interval corresponding to the \((m+1)\)-st occurrence of~\(k+1\) (highlighted above) has length at least~\(1\).
  Let \(\pp'\) be the interval cut for~\(e_{k+1}\) that is obtained from~\(\pp\) by replacing this interval by
  \begin{equation*}
    \cdots \rr{m} p_{2m} \rrbf{k+2} q \rrbf{m+1} q+1 \rrbf{k+2} p_{2m+1} \rr{m+2} \cdots
  \end{equation*}
  for some~\(p_{2m}\le q<p_{2m+1}\). Then
  \begin{equation*}
    \epsilon(\pp') = \epsilon(\pp).
  \end{equation*}
\end{lemma}

\begin{proof}
  One verifies directly that modulo~\(2\) the exponent for the position sign changes by~\(q\), the one for the permutation sign by
  \begin{equation}
    p_{0} + \dots + p_{2m} + q + m + 1,
  \end{equation}
  the one coming from desuspensions by
  \begin{equation}
    p_{0} + \dots + p_{2m} + k + m + 1 
  \end{equation}
  and the one for the explicit sign by~\(k\). Hence there is no sign change in total.
\end{proof}

\section{A bijection}
\label{sec:bijection}

For~\(0\le l\le n\) we define
\begin{align}
  S_{n,l} &= \bigl\{\, \ii=(i_{1},\dots,i_{l})\in\N^{l} \bigm| \text{\(0\le i_{s}\le n-s\) for any~\(1\le s\le l\)}\,\bigr\} \\
  \notag &= \setzero{n-1} \times \setzero{n-2} \times \dots \times \setzero{n-l}
\end{align}
as well as \(S_{n}=S_{n,n}\).
The degree of an element~\(\ii\in S_{n,l}\) is
\begin{equation}
  \deg{\ii} = i_{1}+\dots+i_{l}.
\end{equation}
Note that \(S_{n,n}\) has \(n!\)~elements, and \(S_{n,0}\) has the empty sequence~\(\emptyset\) as unique element.

Let \(1\le k\le n\) and \(\pp=(p_{0},\dots,p_{k})\) where \(0=p_{0}<p_{1}<\dots<p_{k}=n\).
We set \(l=n-k\) and define
\begin{equation}
  \label{eq:def-Snl}
  S_{n-1}(\pp) = \bigl\{\, \ii\in S_{n-1,l} \bigm| \partial_{i_{l}+1}\cdots\partial_{i_{1}+1}\,\setzero{n} = \pp \,\bigr\}.
\end{equation}
Here \(\setzero{n}\) denotes the standard \(n\)-simplex, to which the given face operators are applied in the specified order.
We also set \(q_{s}=p_{s}-p_{s-1}\) for~\(1\le s\le k\).

We define a function
\begin{align}
  \label{eq:def-Psi}
  \Psi_{\pp}\colon S_{n-1}(\pp) &\to \Shuff(q_{1}-1,\dots,q_{k}-1)\times S_{q_{1}-1}\times \dots \times S_{q_{k}-1} \\*
  \notag \ii &\mapsto \bigl(\al=(\alpha_{1},\dots,\alpha_{k}),\jj_{1},\dots,\jj_{k}\bigr)
\end{align}
as follows:
Considering the condition~\eqref{eq:def-Snl},
we think of an element~\(\ii\in S_{n-1}(\pp)\) as describing a way of removing the \(l=n-k\)~elements
not appearing in the sequence~\(\pp\) from the \(n\)-simplex~\(\setzero{n}\).
For~\(1\le s\le k\) the element~\(\jj_{s}\in S_{q_{s}-1}\) similarly records the order
in which the elements between~\(p_{s-1}\) and~\(p_{s}\) are removed by~\(\ii\), ignoring all other removed elements.
The shuffle~\(\al\) keeps track of how the element removals of the intervals~\((p_{s-1},\dots,p_{s})\) are interleaved.
More precisely, we declare \(q-1\in\alpha_{s}\) if and only if
the face operator~\(\partial_{i_{q}+1}\) in~\eqref{eq:def-Snl} removes an element between~\(p_{s-1}\) and~\(p_{s}\).

\begin{example}
  \label{ex:Psi-pp}
  Take \(n=7\), \(k=3\),~\(\pp=(0,3,4,7)\) and~\(\ii=(5,0,0,2)\). The missing elements in~\((0,3,4,7)\)
  are removed in the order~\(6\),~\(1\),~\(2\),~\(5\). Those missing in~\((0,3)\) are removed in the order~\(1\),~\(2\),
  and those missing in~\((4,7)\) in the order~\(6\),~\(5\). We therefore have \(\jj_{1}=(0,0)\), \(\jj_{2}=\emptyset\) and~\(\jj_{3}=(1,0)\)
  as well as \(\alpha_{1}=\{1,2\}\), \(\alpha_{2}=\emptyset\) and~\(\alpha_{3}=\{0,3\}\).
\end{example}

Note that for~\(k=1\) the map\(~\Psi_{\pp}\) boils down to the identity map on~\(S_{n-1}\) because \(\Shuff(n-1)\) is a singleton.
Moreover, for~\(k=2\) we have \(S_{n-1,l}=S_{n-1,n-2}\cong S_{n}\) (since any~\(\ii\in S_{n-1}\) ends in~\(i_{n-1}=0\)),
and the maps~\(\Psi_{\pp}\) with~\(0<p_{1}<n\) combine to the bijection
\begin{equation}
  S_{n-1} \cong \bigcup_{q_{1}+q_{2}=n} \Shuff(q_{1}-1,q_{2}-1)\times S_{q_{1}-1} \times S_{q_{2}-1}
\end{equation}
described by Szczarba~\cite[Lemma~3.3]{Szczarba:1961}.

\begin{proposition}
  \label{thm:Psi-pp-degree}
  The map~\(\Psi_{\pp}\) is bijective, and in the notation of~\eqref{eq:def-Psi} we have
  \begin{equation*}
    \deg{\ii} \equiv (\al)
    + \sum_{s=1}^{k}\deg{\jj_{s}}
    + \sum_{s=1}^{k}(s-1)(q_{s}-1)
    \pmod{2}.
  \end{equation*}
\end{proposition}

Remember from \Cref{sec:shuffle} that given a shuffle~\(\al=(\alpha_{1},\dots,\alpha_{k})\) we write \((\al)\) for the exponent of its signature.
For~\(k=2\) the above identity appears already in~\cite[Lemma~3.3]{Szczarba:1961}
and~\cite[Lemma~6]{HessTonks:2006}.\footnote{%
  Recall that Szczarba writes the signature of the shuffle~\((\nu,\mu)\) as~\(\operatorname{sgn}(\mu,\nu)\)
  and also from~\cite[p.~1866]{HessTonks:2006} that his sign exponent~\(\epsilon(i,n+1)\) equals \(n+\deg{\ii}\).
  Also note that the subscripts of the degeneracy operators~\(s_{\mu}\) and~\(s_{\nu}\) in~\cite{HessTonks:2006} should be swapped.}

\begin{proof}
  It is clear how to reverse the construction to obtain the inverse of~\(\Psi_{\pp}\).
  
  Regarding the claimed formula, we assume first that \(\ii\) is of the form
  \begin{equation}
    \label{eq:ii-first}
    \ii = \bigl(\, \underbrace{0,\dots,0}_{q_{1}-1},\underbrace{1,\dots,1}_{q_{2}-1},\dots,\underbrace{k-1,\dots,k-1}_{q_{k}-1} \,\bigr).
  \end{equation}
  Then the shuffle~\(\al=(\alpha_{1},\dots,\alpha_{k})\) is the identity map on~\(\setzero{l-1}\) and \(\jj_{s}=(0,\dots,0)\) for all~\(s\),
  from which we conclude that the formula holds.

  Consider two elements from~\(\setzero{l-1}\) that are removed one right after the other.
  Changing the order of the removals changes changes the degree of~\(\ii\) by~\(\pm1\).
  If the two removed values belong to the same, say the \(s\)-th, interval of~\(\pp\),
  then the degree of~\(\jj_{s}\) also changes by~\(\pm1\), and \(\al\) remains fixed. If the values belong to different intervals,
  then all~\(\jj_{s}\) remain the same, but the sign of the shuffle changes. Hence in any case the claimed identity
  is preserved.

  Starting from~\eqref{eq:ii-first}, we can reach any~\(\ii\in S_{n-1}(\pp)\)
  by repeating this swapping procedure. This completes the proof.  
\end{proof}

\section{The Szczarba operators}
\label{sec:szczarba}

\subsection{The twisting cochain}

We review the definition of Szczarba's twisting cochain \cite[pp.~200--201]{Szczarba:1961}
in the formulation given by Hess--Tonks~\cite[Sec.~1.4]{HessTonks:2006}.
Let \(X\) be a simplicial set and \(G\) a simplicial group, and let
\begin{equation}
  \tau\colon X_{>0} \to G
\end{equation}
be a twisting function.
It will be convenient in what follows to write \(\sigma(x)=\tau(x)^{-1}\) for~\(x\in X_{>0}\).

Szczarba~\cite[Thm.~2.1]{Szczarba:1961} has introduced the operators
\begin{align}
  \Sz_{\ii} \colon X_{n} &\to G_{n-1} \\
  \notag x &\mapsto \DD{0}{\ii}\,\sigma(x)\,\DD{1}{\ii}\,\sigma(\partial_{0}x)\cdots\DD{n-1}{\ii}\,\sigma((\partial_{0})^{n-1}x)
\end{align}
for~\(n\ge1\) and~\(\ii\in S_{n-1}\). In particular, one has \(\Sz_{\emptyset}x=\sigma(x)\).
We follows Hess--Tonks~\cite[Def.~5]{HessTonks:2006} in using the symbol~\(\Sz_{\ii}\) and the name \newterm{Szczarba operator}.
In terms of these operators, Szczarba's twisting cochain~\(t\colon C(X)\to C(G)\) is given for~\(x\in X_{n}\) by
\begin{equation}
  \label{eq:def-tsz}
  t(x) =
  \begin{cases}
    0 & \text{if \(n = 0\)}, \\
    \Sz_{\emptyset}x-1=\sigma(x)-1 & \text{if \(n = 1\)}, \\
    \sum_{\ii\in S_{n-1}} (-1)^{\deg{\ii}}\,\Sz_{\ii}x  & \text{if \(n \ge 2\)}.
  \end{cases}
\end{equation}
In \Cref{sec:szczarba-degen} we recall the definition of the simplicial operators~\(\DD{k}{\ii}\),
and we show that \(t\) is well-defined on normalized chains.

\begin{example}
  In low degrees, Szczarba's twisting cochain looks as follows. Simplices are indicated by vertex numbers.
  For example, a \(2\)-simplex~\(x\in X_{2}\) is written as~\(012\) and \(s_{1}\,\partial_{0}\,x\) as~\(122\).
  Note that the products are taken in the simplicial group~\(G\), not in the dga~\(C(G)\).
  \def\xx#1{#1}
  \def\icol#1{&\quad\hbox{\small\(\ii\)}&\hbox{\small\(=(#1)\)}}
  \def\icolempty{&\quad\hbox{\small\(\ii\)}&\hbox{\small\(=\emptyset\)}}
  \begin{align}
    t(\xx{01}) &= {} + \sigma(\xx{01})-1, \icolempty \\
    t(\xx{012}) &= {} + \sigma(\xx{012})\,\sigma(\xx{122}),  \icol{0} \\
    t(\xx{0123}) &= {} + \sigma(\xx{0123})\,\sigma(\xx{1223})\,\sigma(\xx{2333}) \icol{0,0} \\
    \notag &\qquad\!\!\!\! {} - \sigma(\xx{0113})\,\sigma(\xx{1233})\,\sigma(\xx{2333}), \icol{1,0} \\
    t(\xx{01234}) &= {} + \sigma(\xx{01234})\,\sigma(\xx{12234})\,\sigma(\xx{23334})\,\sigma(\xx{34444}) \icol{0,0,0} \\
    \notag &\qquad\!\!\!\! {} - \sigma(\xx{01224})\,\sigma(\xx{12224})\,\sigma(\xx{23344})\,\sigma(\xx{34444}) \icol{0,1,0} \\
    \notag &\qquad\!\!\!\! {} - \sigma(\xx{01134})\,\sigma(\xx{12334})\,\sigma(\xx{23334})\,\sigma(\xx{34444}) \icol{1,0,0} \\
    \notag &\qquad\!\!\!\! {} + \sigma(\xx{01114})\,\sigma(\xx{12344})\,\sigma(\xx{23344})\,\sigma(\xx{34444}) \icol{1,1,0} \\
    \notag &\qquad\!\!\!\! {} + \sigma(\xx{01124})\,\sigma(\xx{12224})\,\sigma(\xx{23444})\,\sigma(\xx{34444}) \icol{2,0,0} \\
    \notag &\qquad\!\!\!\! {} - \sigma(\xx{01114})\,\sigma(\xx{12244})\,\sigma(\xx{23444})\,\sigma(\xx{34444}) \icol{2,1,0}
  \end{align}
\end{example}

We need to understand how the Szczarba operators relate to the bijection~\(\Psi_{\pp}\)
introduced in \Cref{sec:bijection}.
Let \(n=k+l\) with~\(1\le k\le n\). We can write any \(\ii=(i_{1},\dots,i_{n-1})\in S_{n-1}\) in the form
\begin{equation}
  \ii = (\ii_{1},\ii_{2}) = (i_{1,1},\dots,i_{1,l},i_{2,1},\dots,i_{2,k-1})
\end{equation}
with~\(\ii_{1}\in S_{n-1}(\pp)\) and~\(\ii_{2}\in S_{k-1}\), where
\begin{equation}
  \pp = (p_{0},\dots,p_{k}) = \partial_{i_{l}+1}\cdots\partial_{i_{1}+1}\,\setzero{n}.
\end{equation}

\begin{lemma}
  \label{thm:Sz-Psi}
  Using this notation, we have
  \begin{align*}
    (\partial_{0})^{l} \Sz_{\ii} x &= \Sz_{\ii_{2}}x(p_{0},p_{1},\dots,p_{k}), \\
    \tpartial^{k-1}\Sz_{\ii} x &= s_{\bar\alpha_{1}}\Sz_{\jj_{1}}x(p_{0},\dots,p_{1})\cdots s_{\bar\alpha_{k}}\Sz_{\jj_{k}}x(p_{k-1},\dots,p_{k})
  \end{align*}
  where \(\Psi_{\pp}(\ii_{1})=(\al,\jj_{1},\dots,\jj_{k})\) and \(\bar\alpha_{s}=\setzero{l-1}\setminus\alpha_{s}\) for~\(1\le s\le k\).
\end{lemma}

\begin{proof}
  The case~\(l=0\) of the first identity is void. Given the definition~\eqref{eq:def-Snl}, it reduces for~\(l=1\) to the formula
  \begin{equation}
    \partial_{0} \Sz_{\ii} x = \Sz_{(i_{2},\dots,i_{n-1})} \partial_{i_{1}+1} x,
  \end{equation}
  which is stated in~\cite[Lemma~6]{HessTonks:2006}. The case~\(l\ge 2\) follows by iteration.

  The second identity is trivial for~\(k=1\), compare the discussion of~\(\Psi_{\pp}\) following \Cref{ex:Psi-pp}.
  For~\(k=2\) it is again given in~\cite[Lemma~6]{HessTonks:2006}. For larger~\(k\) it follows by induction:

  Assume the identity proven for~\(k\) and~\(l\) and consider \(k'=k+1\) and~\(l'=l-1\).
  The other values for the new situation are also written with a prime, that is, \(\pp'\), \(\ii'=(\ii_{1}',\ii_{2}')\)
  and \(\Psi_{\pp'}(\ii_{1}')=(\al',\jj_{1}',\dots,\jj_{2}')\).

  Let \(\pp=\partial_{i_{2,1}'+1}\pp'\), and let \(\hat p\) be the removed value.
  We split \(\ii'\) as~\(\ii'=(\ii_{1},\ii_{2})\) with~\(\ii_{1}=(i_{1,1}',\dots,i_{1,l-1}',i_{2,1}')\) and~\(\ii_{2}=(i_{2,2}',\dots,i_{2,k-1}')\)
  and corresponding values~\(\al\) and~\(\jj_{1}\),~\dots,~\(\jj_{k}\). Then
  \begin{align}
    \MoveEqLeft{\tpartial^{k}\Sz_{\ii'} x = \tpartial\,\tpartial^{k-1}\Sz_{\ii'} x} \\*
    \notag & = \tpartial \Bigl( s_{\bar\alpha_{1}}\Sz_{\jj_{1}}x(p_{0},\dots,p_{1})\cdots s_{\bar\alpha_{k}}\Sz_{\jj_{k}}x(p_{k-1},\dots,p_{k}) \Bigr) \\
  \intertext{By the definition of the shuffle~\(\al\), we have \(l-1\in\alpha_{r}\) if \(\partial_{i_{l}+1}\) removes an element
    between~\(p_{r-1}\) and~\(p_{r}\). Hence \(l-1\notin\alpha_{s}\) for~\(s\ne r\) and therefore}
    \notag & = s_{\bar\alpha_{1}\setminus\{l-1\}}\Sz_{\jj_{1}}x(p_{0},\dots,p_{1}) \\
    \notag &\qquad\qquad \cdots s_{\bar\alpha_{r}}\tpartial\Sz_{\jj_{r}}x(p_{r-1},\dots,p_{r}) \cdots s_{\bar\alpha_{k}\setminus\{l-1\}}\Sz_{\jj_{k}}x(p_{k-1},\dots,p_{k}).
  \end{align}
  Set \(\hat q_{1}=\hat p-p_{r-1}\) and~\(\hat q_{2}=p_{r}-\hat p\). Again by the case~\(l=2\) we have
  \begin{equation}
    \tpartial\Sz_{\jj_{r}}x(p_{r-1},\dots,p_{r}) = s_{\bar\beta_{2}}\Sz_{\hjj_{1}}x(p_{r-1},\dots,\hat p)\cdot s_{\bar\beta_{1}}\Sz_{\hjj_{2}}x(\hat p,\dots,p_{r})
  \end{equation}
  for a \((\hat q_{1}-1,\hat q_{2}-1)\)-shuffle~\((\beta_{1},\beta_{2})\)
  and sequences~\(\hjj_{1}\in S_{\hat q_{1}-1}\), \(\hjj_{2}\in S_{\hat q_{2}-1}\).
  We thus obtain the desired formula since
  \begin{gather}
    \jj_{1}' = \jj_{1}, \;\; \dots, \;\;
    \jj_{r}'=\hjj_{1}, \;\; \jj_{r+1}'=\hjj_{2},
    \;\; \dots, \;\; \jj_{k+1}' = \jj_{k}, \\
    \al' = (\alpha_{1},\dots,\alpha_{r-1},\gamma_{1},\gamma_{2},\alpha_{r+1},\dots,\alpha_{k})
  \end{gather}
  where the subsets~\(\gamma_{1}\),~\(\gamma_{2}\subset\setzero{l-2}\) are defined by
  \begin{equation}
    s_{\bar\gamma_{1}} = s_{\bar\alpha_{r}}\,s_{\bar\beta_{1}},
    \qquad
    s_{\bar\gamma_{2}} = s_{\bar\alpha_{r}}\,s_{\bar\beta_{2}}.
    \qedhere
  \end{equation}
\end{proof}

\subsection{The twisted shuffle map}

Let \(F\) be a left \(G\)-space.
We recall the definition of Szczarba's twisted shuffle map \cite[Thm.~2.3]{Szczarba:1961}
\begin{equation}
  \label{eq:def-psi}
  \psi=\psi_{F}\colon C(X)\otimes_{t}C(F) \to C(X\times_{\tau}F)
\end{equation}
in a notation inspired by Hess--Tonks.
For any~\(n\ge0\) and~\(\ii\in S_{n}\) we define the operator
\begin{align}
  \hatSz_{\ii}\colon X_{n} &\to (X\times_{\tau}G)_{n} = X_{n} \times G_{n}, \\
  \notag x &\mapsto \bigl( \DD{0}{\ii}\,x, \DD{1}{\ii}\,\sigma(x)\,\DD{2}{\ii}\,\sigma(\partial_{0}x)\cdots\DD{n}{\ii}\,\sigma((\partial_{0})^{n-1}x) \bigr),
\end{align}
which is interpreted as~\(\hatSz_{\emptyset}x=(x,1)\in X_{0}\times G_{0}\) for~\(n=0\) and~\(\ii=\emptyset\). Based on this we define the map
\begin{equation}
  \psi(x\otimes y) = \sum_{\ii\in S_{n}}\,(-1)^{\deg{\ii}}\,(\id_{X},\mu_{F})_{*}\,\shuffle\,\bigl(\,\hatSz_{\ii}x\otimes y\,\bigr),
\end{equation}
where \(n=\deg{x}\) as before, \(\shuffle\colon C(X\times_{\tau}G)\otimes C(F)\to C(X\times_{\tau}G\times F)\) is the shuffle map
and \(\mu_{F}\colon G\times F\to F\) the group action.\footnote{%
  In the definition of~\(\psi\) in~\cite[p.~201]{Szczarba:1961}
  the upper summation index should read ``\(p!\)''.}
For a proof that \(\psi\) descends to normalized chains see again~\Cref{sec:szczarba-degen}.

Given a decomposition \(n=k+l\) with~\(k\),~\(l\ge0\), we can write any~\(\ii\in S_{n}\) in the form~\(\ii=(\ii_{1},\ii_{2})\)
with~\(\ii_{1}\in S_{n}(\pp)\) and~\(\ii_{2}\in S_{k}\), where
\begin{equation}
  \pp = (0=p_{0},p_{1},\dots,p_{k+1}=n+1) = \partial_{i_{l}+1}\cdots\partial_{i_{1}+1}\,\setzero{n+1}.
\end{equation}
We also write \(q_{s}=p_{s}-p_{s-1}\) for~\(1\le s\le k+1\).

\begin{lemma}
  In the notation above, we have
  \begin{align*}
    (\partial_{0})^{l}\hatSz_{\ii} x &= \hatSz_{\ii_{2}} x(p_{1}-1,\dots,p_{k+1}-1), \\
    \tpartial^{k}\,\hatSz_{\ii} x &= s_{\bar\alpha_{1}}\hatSz_{\jj_{1}} x(0,\dots,p_{1}-1)
    \cdot s_{\bar\alpha_{2}}\Sz_{\jj_{2}} x(p_{1}-1,\dots,p_{2}-1) \\
    &\qquad\qquad\qquad\qquad\quad \cdots s_{\bar\alpha_{k+1}}\Sz_{\jj_{k+1}} x(p_{k}-1,\dots,p_{k+1}-1),
  \end{align*}
  where \(\Psi_{\pp}(\ii_{1})=(\al,\jj_{1},\dots,\jj_{k+1})\) and \(\bar\alpha_{s}=\setzero{l-1}\setminus\alpha_{s}\) for~\(1\le s\le k+1\).
\end{lemma}

\begin{proof}
  Apart from the trivial case~\(l=0\), the first formula follows by induction from the case~\(l=1\), that is,
  \begin{equation}
    \partial_{0}\,\hatSz_{\ii} x = \hatSz_{(i_{2},\dots,i_{n})} \partial_{i_{1}}\,x,
  \end{equation}
  which can be found in~\cite[pp.~205--206]{Szczarba:1961}
  as the discussion of the ``first term of~(4.1)'' there.

  The second formula is also trivial for~\(k=0\), and for~\(k=1\) it is contained in~\cite[eq.~(4.5)]{Szczarba:1961}.
  The extension to larger~\(k\) follows again by induction, based
  on the case~\(k=2\) of the present claim as well as the case~\(k=2\) of \Cref{thm:Sz-Psi}, using the same kind of reasoning as given there.
\end{proof}

\section{Proof of Theorem~\ref{thm:main}}
\label{sec:proof-main}

Let \(X\) be a simplicial set, \(G\) a simplicial group and \(\tau\colon X_{>0}\to G\) a twisting function.
Explicitly, the Szczarba map~\eqref{eq:OMCX-CG} is given by
\begin{equation}
  \Sz\colon \OM \,C(X) \to C(G),
  \qquad
  \bigCobarEl{ x_{1} \mid \cdots \mid x_{m} } \mapsto t(x_{1})\cdots t(x_{m})
\end{equation}
where \(t\colon C(X)\to C(G)\) is Szczarba's twisting cochain as defined in~\eqref{eq:def-tsz}.
Since we are looking at a multiplicative map between bialgebras, we only have to show
\begin{equation}
  \label{eq:main-ts}
  \Delta_{C(G)}\,\Sz\,\CobarEl{x} = (\Sz\otimes\Sz)\,\Delta_{\OM C(X)}\,\CobarEl{x}
\end{equation}
for any~\(x\in X\), say of degree~\(n\).
If \(n=0\), then \(\CobarEl{x}\) is primitive and annihilated by~\(\Sz\), so that \eqref{eq:main-ts} holds.
We therefore assume \(n\ge1\) for the rest of the proof.

The left-hand side of~\eqref{eq:main-ts} equals
\begin{align}
  \label{eq:main-ts-lhs}
  \Delta\,\Sz\,\CobarEl{x} &= \Delta\,t(x) = \sum_{k=1}^{n}\tpartial^{k-1}\,t(x)\otimes(\partial_{0})^{l}\,t(x) \\*
  \notag &= \sum_{k=1}^{n}\,\sum_{\ii\in S_{n-1}}(-1)^{\deg{\ii}}\,\tpartial^{k-1}\,\Sz_{\ii}x\otimes(\partial_{0})^{l}\,\Sz_{\ii}x
\end{align}
where we have again used the abbreviation~\(l=n-k\).
Using the explicit formula~\eqref{eq:def-diag-cobar-CX} for the diagonal, we can write
the right-hand side of~\eqref{eq:main-ts} in the form
\begin{equation}
  \label{eq:main-ts-rhs}
  (\Sz\otimes\Sz)\,\Delta\,\CobarEl{x} = t(x)\otimes 1 + \sum_{k=0}^{n} \sum_{\pp}\, (-1)^{\epsilon(\pp)}\,t(x^{\pp}_{1})\cdots t(x^{\pp}_{k})\otimes t(x^{\pp}_{k+1})
\end{equation}
where \(\pp=(p_{0},p_{1},\dots,p_{2k+1})\) ranges over the cuts of~\(\setzero{n}\) into \(2k+1\)~intervals corresponding to the surjection~\(e_{k}\).
We are going to pair off the summands of the expressions~\eqref{eq:main-ts-lhs} and~\eqref{eq:main-ts-rhs}.
We write \(q_{s}=p_{2s}-p_{2s-1}\) for~\(1\le s\le k\) and
\(\ell(\pp)\) for the sum of the lengths of the intervals in~\(\pp\) corresponding to the final value~\(k+1\).

Assume \(\ell(\pp)=0\), so that the \(k\)~intervals labelled \(1\),~\dots,~\(k\) cover the whole interval~\(\setzero{n}\).
From the definition of~\(t\) we get
\begin{equation}
  \label{eq:txppk1}
  t(x^{\pp}_{k+1}) = \sum_{\ii_{2}\in S_{k-1}} (-1)^{\deg{\ii_{2}}}\,\Sz_{\ii_{2}} x^{\pp}_{k+1},
\end{equation}
and together with that of the shuffle map~\eqref{eq:def-shuffle} also
\begin{multline}
  \label{eq:txpp1k}
  (-1)^{\epsilon(\pp)}\,t(x^{\pp}_{1})\cdots t(x^{\pp}_{k}) = \\
  \sum (-1)^{\epsilon(\pp) + (\al) + \sum_{s}\deg{\jj_{s}}}
  \, s_{\bar\alpha_{1}} \Sz_{\jj_{1}} x^{\pp}_{1}
  \cdots s_{\bar\alpha_{k}} \Sz_{\jj_{k}} x^{\pp}_{k} \\
  {} + \text{additional terms with fewer than~\(k\) factors.}
\end{multline}
Here the sum is over all \((q_{1}-1,\dots,q_{k}-1)\)-shuffles~\(\al=(\alpha_{1},\dots,\alpha_{k})\)
as well as over all~\(\jj_{1}\in S_{q_{1}-1}\),~\dots,~\(\jj_{k}\in S_{q_{k}-1}\).
The additional terms indicated above arise whenever we have~\(q_{s}=1\) for some~\(s\)
because of the extra term~\(-1\in C(G)\) produced by~\(t\) in the case of a \(1\)-simplex.

Consider the case~\(k>1\).
According to \Cref{thm:Sz-Psi},
the expressions~\((\partial_{0})^{l}\,t(x)\) in~\eqref{eq:main-ts-lhs} that give terms of the form~\eqref{eq:txppk1}
are indexed by the~\(\ii=(\ii_{1},\ii_{2})\in S_{n-1}\) with~\(\ii_{1}\in S_{n-1}(\pp)\) and~\(\ii_{2}\in S_{k-1}\).
By the same lemma, the terms~\(\tpartial^{k-1}\,t(x)\)
for all such~\(\ii_{1}\) give exactly the terms in the sum formula of~\eqref{eq:txpp1k}.
\Cref{eq:epp-length0} and~\Cref{thm:Psi-pp-degree} show that also the signs work out correctly since
\(\deg{\ii}=\deg{\ii_{1}}+\deg{\ii_{2}}\) and
\begin{equation}
  \deg{\ii_{1}}
  = (\al) + \sum_{s=1}^{k} \deg{\jj_{s}} + \sum_{s=1}^{k}(s-1)(q_{s}-1)
  = \epsilon(\pp) + (\al) + \sum_{s=1}^{k} \deg{\jj_{s}}.  
\end{equation}

If \(k=1\), then \(x^{\pp}_{1}=x\), ~\(x^{\pp}_{2}=x(0,n)\) is of degree~\(1\), and~\(\epsilon(\pp)=0\).
In addition to the terms discussed in the preceding paragraph,
we get a~\(-1\) on the right-hand side of~\eqref{eq:txppk1} and therefore \(-t(x) \otimes 1\) in~\eqref{eq:main-ts-rhs},
which cancels with the very first term in the same formula.

We now argue that the decompositions~\(\pp\) with~\(\ell(\pp)>0\)
(including the only possible decomposition for~\(k=0\)) lead to terms in the sum~\eqref{eq:main-ts-rhs} that
cancel out with the additional terms in~\eqref{eq:txpp1k} for~\(\ell(\pp)=0\).

Given two decompositions~\(\pp\) and~\(\pp'\) for the surjections~\(e_{k}\) and~\(e_{k'}\), we write \(\pp'\ge\pp\)
if \(\pp'\) can be obtained from~\(\pp\) by zero or more applications of the ``refinement procedure'' described in \Cref{thm:pp-refine}.
This gives a partial order on the set of all such decompositions.

For any decomposition~\(\pp\) there are exactly \(2^{\ell(\pp)}\)~decompositions~\(\pp'\ge\pp\).
In the maximal such~\(\pp'\), all intervals with the final label~\(k+1\) in~\(\pp\) have been subdivided into intervals of length~\(1\)
and relabelled with non-final values, separated by intervals of length~\(0\) labelled~\(k+1\). In particular, \(\ell(\pp')=0\).
Conversely, there are exactly \(2^{\ell_{1}(\pp)}\)~decompositions~\(\pp'\le\pp\), where \(\ell_{1}(\pp)\) is the number of intervals
of length~\(1\) having non-final labels. The minimal such~\(\pp'\) has no intervals of this kind.

\begin{example}
  Take \(k=1\) and the decomposition
  \begin{equation}
    \pp\colon 0 \rr{2} 0 \rr{1} 1 \rr{2} 3.
  \end{equation}
  The maximal~\(\pp'\ge\pp\) and the minimal~\(\pp''\le\pp\) are as follows.
  Subdivided or combined intervals are indicated in boldface.
  \begin{gather}
    \pp'\colon 0 \rr{4} 0 \rr{1} 1 \rr{\mathbf{4}} 1 \rr{\mathbf{2}} 2 \rr{\mathbf{4}} 2 \rr{\mathbf{3}} 3 \rr{\mathbf{4}} 3 \qquad (k'=3), \\
    \pp''\colon 0 \rr{\mathbf{1}} 3 \qquad (k''=0).
  \end{gather}
\end{example}

Note that we have
\begin{equation}
  x^{\pp}_{k+1} = x^{\pp'}_{k'+1}
\end{equation}
whenever \(\pp\) and~\(\pp'\) are comparable.
We therefore look at a minimal~\(\pp\) in our ordering and the term~\(x^{\pp}_{k+1}\) it produces.
As the added intervals of any~\(\pp'\ge\pp\) are all of length~\(1\),
the corresponding terms~\(t(x^{\pp'}_{s})\) in
\begin{equation}
  \label{eq:txpp1k-prime}
  (-1)^{\epsilon(\pp')}\,t(x^{\pp'}_{1})\cdots t(x^{\pp'}_{k'})
\end{equation}
all contain \(-1\in C(G)\). The summand
\begin{equation}
  \label{eq:ell-pp-a}
  (-1)^{\epsilon(\pp') + (\al') + \sum_{s'}\deg{\jj'_{s'}} + \ell_{1}(\pp')}\,
  \prod_{\substack{1\le s'\le k'\\q'_{s'}\ne 1}} s_{\bar\alpha_{s'}} \Sz_{\jj'_{s'}} x^{\pp'}_{s'}
  =\vcentcolon (-1)^{\ell_{1}(\pp')}\,a
\end{equation}
therefore appears in the product~\eqref{eq:txpp1k-prime}.
We claim that the expression~\(a\) only depends on~\(\pp\). More precisely, we have
\begin{equation}
  a = (-1)^{\epsilon(\pp) + (\al) + \sum_{s}\deg{\jj_{s}}}\,
  s_{\bar\alpha_{1}} \Sz_{\jj_{1}} x^{\pp}_{1}
  \cdots s_{\bar\alpha_{k}} \Sz_{\jj_{k}} x^{\pp}_{k}.
\end{equation}
This is because an interval of length~\(q'_{s'}=1\) leads to~\(\alpha'_{s'}=\emptyset\) and~\(\jj'_{s'}=\emptyset\),
while the remaining~\(\alpha'_{s'}\) and~\(\jj'_{s'}\) are not affected and appear as~\(\alpha_{s}\) and~\(\jj_{s}\)
for some index~\(s\le s'\).
Moreover, we have \(\epsilon(\pp')=\epsilon(\pp)\) by a repeated application of \Cref{thm:pp-refine}.

If \(\ell(\pp)>0\), then we get \(2^{\ell(\pp)}\)~terms with alternating signs, so that
\begin{equation}
  \sum_{\pp'\ge\pp} (-1)^{\ell_{1}(\pp')}\,a \otimes t(x^{\pp'}_{k'+1})
  = \sum_{\pp'\ge\pp} (-1)^{\ell_{1}(\pp')}\,a \otimes t(x^{\pp}_{k+1}) = 0.
\end{equation}
The only terms in~\eqref{eq:main-ts-rhs} not appearing in such a sum are \(t(x)\otimes 1\)
plus those written out in~\eqref{eq:txpp1k} for~\(\pp\) with~\(\ell(\pp)=0\),
and we have seen already that they add up to~\eqref{eq:main-ts-lhs}.
This completes the proof.

\section{The extended cobar construction and the loop group}
\label{sec:cobar}

Let \(X\) be a reduced simplicial set (that is, having a unique \(0\)-simplex), and let
\(\GG X\) be its Kan loop group, compare~\cite[Def.~26.3]{May:1968}.
(Its topological realization~\(|\Kan X|\) is a model for the based loop space~\(\Omega|X|\)
as a topological monoid, see~\cite[\S 1.8, Prop.~3.3]{Berger:1995}.)
Let \(\tau\colon X_{>0}\to \GG X\) be the canonical twisting function, and
let \(t\) be Szczarba's twisting cochain associated to it.

Hess--Tonks have defined an extended cobar construction~\(\OMex\,C(X)\) such that
the canonical dga map~\(\OM\,C(X)\to C(\GG X)\) extends to a dga map
\begin{equation}
  \label{eq:def-phi}
  \phi\colon \OMex\,C(X)\to C(\GG X),
\end{equation}
see~\cite[Thm.~7]{HessTonks:2006}.
They moreover showed that \(\phi\) is a strong deformation retract of chain complexes
such that all maps involved are natural in~\(X\) \cite[Thm.~15]{HessTonks:2006}.

Let us recall the definition of~\(\OMex\,C(X)\) in the form given by Rivera--Saneblidze~\cite[Sec.~4.2]{RiveraSaneblidze:2019}.
Write \(C=C(X)\), and let \(G\) be the free group on generators~\(g_{x}\) where \(x\) runs through
the non-degenerate \(1\)-simplices of~\(X\). We define a new dgc~\(\Cex\) by~\(\Cex_{n}=C_{n}\) for~\(n\ne1\)
and~\(\Cex_{1}=\kk[G]\), the group algebra of~\(G\). We set \(d\,g=0\), \(\epsilon(g)=0\)
and~\(\Delta\,g=g\otimes 1_{C}+1_{C}\otimes g\) for any~\(g\in G\).
We embed \(C\) into~\(\Cex\) by sending \(x\)~as before to~\(g_{x}-1_{G}\).
The dga~\(\OMex\,C(X)\) is the quotient of the usual cobar construction~\(\OM\,\Cex\)
by the two-sided dg~ideal generated by the cycles~\(\CobarEl{a|b} - \CobarEl{a b}\) for~\(a\),~\(b\in\Cex_{1}\)
as well as~\(\CobarEl{1_{G}}-\smash{1_{\OM\Cex}}\).
By abuse of notation, we write elements of~\(\OMex\,C(X)\) like those of~\(\OM\,\Cex\).

We extend Szczarba's twisting cochain~\(t\) to a linear map~\(\tex\colon \Cex\to C(\GG X)\)
by defining \(\tex(g_{x})=\sigma(x)\) for any non-degenerate \(1\)-simplex~\(x\in X\) and
taking its multiplicative extension to~\(G\subset\Cex_{1}\). The result is again a twisting cochain.
The induced dga morphism~\(\OM\,\Cex\to C(\GG X)\) descends to~\(\OMex\,C(X)\), where it defines the map~\(\phi\) from~\eqref{eq:def-phi}.

We extend the augmentation and the diagonal from~\(\OM\,C(X)\) to~\(\OM\,\Cex\) by setting
\begin{equation}
  \label{eq:def-OMex-dgc}
  \epsilon(\CobarEl{g}) = 1
  \qquad\text{and}\qquad
  \Delta\,\CobarEl{g} = \CobarEl{g} \otimes \CobarEl{g}
\end{equation}
for any~\(g\in G\). This induces well-defined maps on~\(\OMex\,C(X)\).

\begin{proposition}
  \label{thm:excobar-comult}
  Let \(X\) be a reduced simplicial set.
  With the structure maps given above, \(\OMex\,C(X)\) becomes a dg~bialgebra
  and \(\phi\) a quasi-isomorphism of dg~bialgebras.
\end{proposition}

\begin{proof}
  The maps~\eqref{eq:def-OMex-dgc} are compatible with~\(\phi\) because analogous formulas hold for the \(0\)-simplices~\(\phi(\CobarEl{g})\in\GG X\).
  Since \(\phi\) is a deformation retract, it is an injective quasi-iso\-mor\-phism and its image a direct summand of~\(C(\GG X)\).
  Because the latter is a dg~bialgebra, so is \(\OMex\,C(X)\), and \(\phi\) is a morphism of dg~bialgebras.
\end{proof}

\begin{remark}
  The extended cobar construction~\(\OMex\,C(X)\) is in fact the normalized chain complex
  of a certain cubical monoid~\(Y=\OMex X\), see~\cite[Sec.~3.5]{RiveraSaneblidze:2019}.
  This cubical monoid can be (formally) triangulated to a simplicial monoid~\(\Tri\,Y\).
  Sending each \(n\)-cube to the \(n!\)~simplices in its triangulation gives a well-defined
  quasi-isomorphism of dg~bialgebras~\(\TT\colon C(Y)\to C(\Tri\,Y)\).
  After the prepublication of this article,
  Minichiello--Rivera--Zeinalian~\cite[Cor.~5.20]{MinichielloRiveraZeinalian:2022} have shown that
  there is a morphism of simplicial monoids~\(f\colon\Tri\,Y\to GX\) such that
  \(\phi = f_{*}\circ\TT\). This gives a different proof that \(\phi\) is morphism of dg~bialgebras.
\end{remark}

\section{Twisted tensor products}
\label{sec:twisted-tensor}

Let \(C\) be an hgc and \(A\) a dg~bialgebra, and let \(M\) be an \(A\)-dgc.
By the latter we mean a dgc~\(M\) that is also a left \(A\)-module
such that the diagonal~\(\Delta_{M}\colon M\to M\otimes M\) and the augmentation~\(\epsilon_{M}\colon M\to\kk\)
are \(A\)-equivariant. (Recall that \(A\) acts on~\(M\otimes M\) via its diagonal~\(\Delta_{A}\colon A\to A\otimes A\)
and on~\(\kk\) via its augmentation~\(\epsilon_{A}\colon A\to\kk\).)

Let \(t\colon C\to A\) be a twisting cochain.
The differential of the twisted tensor product~\(C\otimes_{t} M\) is given by
\begin{equation}
  d_{t} = d_{C}\otimes 1 + 1\otimes d_{M} - \delta_{t}
\end{equation}
where
\begin{equation}
  \label{eq:def-delta-t}
  \delta_{t} = (1\otimes\mu_{M})\,(1\otimes t\otimes 1)\,(\Delta_{C}\otimes 1)
\end{equation}
and \(\mu_{M}\colon A\otimes M\to M\) is the structure map of the \(A\)-module~\(M\).
In the Sweedler notation this is expressed as
\begin{equation}
  d_{t}(c\otimes m) = d\,c\otimes m + (-1)^{\deg{c}}\,c\otimes d\,m - \sum_{(c)} (-1)^{\deg{\swee{c}{1}}}\,\swee{c}{1}\otimes t(\swee{c}{2})\,m
\end{equation}
for~\(c\otimes m\in C\otimes_{t}M\).

The purpose of this section is to observe that \(C\otimes_{t} M\) can again be turned into a dgc
if \(t\) is comultiplicative.
The dual situation of a multiplication on the twisted tensor product of an hga and a dg~bialgebra
has already been considered by Kadeishvili--Saneblidze~\cite[Thm.~7.1]{KadeishviliSaneblidze:2005}.

Let \(f\colon\OM\,C\to A\) be the map of dg~bialgebras induced by the comultiplicative twisting cochain~\(t\).
Based on~\(f\) and on the map~\(\EEE\) from~\eqref{eq:def-EEE}, we introduce the map of degree~\(0\)
\begin{equation}
  \FFF\colon C \stackrel{\EEE}{\longrightarrow} \OM\,C\otimes C \xrightarrow{f\otimes 1} A\otimes C.
\end{equation}
The diagonal of~\(C\otimes_{t} M\) then is defined as
\begin{multline}
  \label{eq:def-diag-twisted}
  \qquad
  \Delta =
  (1_{C}\otimes\mu_{M}\otimes 1_{C}\otimes 1_{M})(1_{C}\otimes 1_{A}\otimes T_{C,M}\otimes 1_{M})\, \\*
  \bigl(1_{C}\otimes \FFF\otimes 1_{M}\otimes 1_{M}\bigr)\,(\Delta_{C}\otimes\Delta_{M})
  \qquad
\end{multline}
where \(\mu_{M}\colon A\otimes M\to M\) is the action. In terms of the Sweedler notation this means
\begin{equation}
  \Delta(c\otimes m) = \sum_{(c), (m)} \sum_{i}\, (-1)^{\deg{c_{i}}\deg{\swee{m}{1}}}\,
  \bigl( \swee{c}{1} \otimes a_{i}\cdot\swee{m}{1} \bigr) \otimes \bigl( c_{i} \otimes \swee{m}{2} \bigr)
\end{equation}
for~\(c\otimes m\in C\otimes_{t}M\) and \(\FFF(\swee{c}{2})=\sum_{i} a_{i}\otimes c_{i}\in A\otimes C\).

\begin{proposition}
  Let \(t\colon C\to A\) be a comultiplicative twisting cochain and \(M\) an \(A\)-dgc.
  Then the twisted tensor product~\(C\otimes_{t} M\) is a dgc
  with the diagonal given above and the augmentation~\(\epsilon_{C}\otimes\epsilon_{M}\).
\end{proposition}

\begin{proof}
  This is a lengthy computation based on the analogues
  \begin{gather}
    \label{eq:prop-FFF-1}
    d(\FFF) = (\mu_{A} \otimes 1_{C})(t \otimes \FFF)\,\Delta_{C}
    - (\mu_{A} \otimes 1_{C})\,(1_{A}\otimes T_{C,A})\,(\FFF \otimes t)\,\Delta_{C}, \\
    \label{eq:prop-FFF-2}
    (1_{A} \otimes \Delta_{C})\,\FFF = (\mu_{A} \otimes 1_{C} \otimes 1_{C})\,(1_{A} \otimes T_{C,A} \otimes 1_{C})\,(\FFF \otimes \FFF)\,\Delta_{C}, \\
    \label{eq:prop-FFF-3}
    (\Delta_{A} \otimes 1_{C})\,\FFF = (1_{A} \otimes \FFF)\,\FFF.
  \end{gather}
  of the identities for~\(\EEE\) stated in \Cref{thm:conditions-EEE}.
  One additionally uses the formula
  \begin{equation}
    \label{eq:Delta-t}
    \Delta_{A}\,t = (1\otimes t)\,\FFF + t\otimes\iota_{A},
  \end{equation}
  which can be seen as follows:
  Since \(f\) is a morphism of coalgebras, one has
  \begin{equation}
    \Delta_{A}\,t=\Delta_{A}\,f\,t_{C}=(f\otimes f)\,\Delta_{\OM C}\,t_{C}=(f\otimes f)\,\EE.
  \end{equation}
  The image of~\(\EE\) lies in~\(\OM\,C\otimes\OM_{l}\,C\) with~\(l\le1\). Considering the terms for~\(l=0\) and~\(l=1\)
  separately as in the proof of \Cref{thm:conditions-EEE} gives \eqref{eq:Delta-t}.

  In order to prove that \(\Delta=\Delta_{C\otimes M}\) as given in~\eqref{eq:def-diag-twisted} is a chain map,
  it is convenient to use the tensor product differential~\(d_{\otimes}=d_{C}\otimes 1+1\otimes d_{M}\) on~\(C\otimes M\)
  and analogously on~\((C\otimes M)\otimes(C\otimes M)\)
  and to show that
  \begin{equation}
    d_{\otimes}(\Delta_{C\otimes M})
    - (\delta_{t}\otimes 1_{C\otimes M})\,\Delta_{C\otimes M} - (1_{C\otimes M}\otimes \delta_{t})\,\Delta_{C\otimes M}
    + \Delta_{C\otimes M}\,\delta_{t} = 0.
  \end{equation}
  With respect to these differentials,
  \(\FFF\) is the only map appearing in~\eqref{eq:def-diag-twisted} that is not a chain map.
  The boundary~\(d_{\otimes}(\Delta)\) therefore has two summands coming from the right-hand side of~\eqref{eq:prop-FFF-1}.
  The first of them cancels with~\((\delta_{t}\otimes 1)\,\Delta\).
  Using \eqref{eq:Delta-t}, the term~\(\Delta\,\delta_{t}\) splits up into two. Taking \eqref{eq:prop-FFF-2} into account,
  the first one cancels with~\((1\otimes\delta_{t})\,\Delta\) and the second one with the second summand in~\(d_{\otimes}(\Delta)\).

  The coassociativity of~\(\Delta_{C\otimes M}\) is a consequence of~\eqref{eq:prop-FFF-2} and~\eqref{eq:prop-FFF-3}.
  The properties involving the augmentation follow directly from the definitions.
\end{proof}

\begin{corollary}
  Let \(t\colon C(X)\to C(G)\) be Szczarba's twisting cochain determined by a twisting function~\(\tau\colon X_{>0}\to G\),
  and let \(F\) be a left \(G\)-space. Then \(C(X)\otimes_{t}C(F)\) is a dgc.
\end{corollary}

The diagonal is independent of the chosen coaugmentation of~\(C(X)\) and looks explicitly as follows:
For~\(x\in X_{n}\) and~\(y\in F_{m}\) we have
\begin{multline}
  \label{eq:diag-twisted-tensor-explicit}
  \Delta\,(x\otimes y) =
  \sum_{i=0}^{n}\sum_{j=0}^{m}\sum_{k=0}^{n-i}\sum_{\pp}\,
  (-1)^{\epsilon(\pp)+i+(m-j-1)\deg{z^{\pp}_{k+1}}} \\*
  \cdot \Bigl(\tpartial^{i}\,x \otimes t(z^{\pp}_{1})\cdots t(z^{\pp}_{k})\, 
  \tpartial^{\,j}\,y \Bigr)
  \otimes \Bigl( z^{\pp}_{k+1}\otimes (\partial_{0})^{m-j}\,y \Bigr)
\end{multline}
where \(z=(\partial_{0})^{n-i}\,x\), and the last sum is over all interval cuts~\(\pp\) of~\(\setzero{i}\) corresponding to~\(e_{k}\).
(Recall that the unit~\(1\in\OM\,C\) is annihilated by the map~\(p_{C}\) implicit in~\(\FFF\) and defined in~\eqref{eq:dec-pC},
hence so is the term~\(\CobarEl{z}\otimes1\) appearing in~\(\Delta\,\CobarEl{z}\) by~\(1\otimes p_{C}\).)

\section{Proof of Theorem~\ref{thm:psi-dgc}}
\label{sec:morph-dgc}

This proof is similar to the one for \Cref{thm:main} given in~\Cref{sec:proof-main}.
Since Szczarba proved that \(\psi_{F}\) is a chain map \cite[Thm.~2.4]{Szczarba:1961},
we only need to show that \(\psi_{F}\) is a morphism of coalgebras.
We start by observing that it is enough to consider the case~\(F=G\) because
we can write the twisted shuffle map~\(\psi_{F}\) in the form
\begin{multline}
  C(X)\otimes_{t}C(F) = C(X)\otimes_{t}C(G) \otimesunder{C(G)} C(F) \\
  \xrightarrow{\psi_{G}\otimes1} C(X\times_{\tau}G) \otimesunder{C(G)} C(F)
  \xrightarrow{\;\shuffle\;} C\bigl(X\times_{\tau}G \timesunder{G} F\bigr) = C(X\times_{\tau}F).
\end{multline}
Hence if \(\psi_{G}\) is a dgc map, then so is \(\psi_{F}\).
(Recall from~\cite[(17.6)]{EilenbergMoore:1966} that the shuffle map~\(\shuffle\) is a morphism of dgcs.
This also implies that the tensor product of a left and a right \(A\)-dgc
over a dg~bialgebra~\(A\) is again a dgc, compare~\cite[p.~848]{FelixHalperinThomas:1995}.)

The diagonal on the right \(C(G)\)-module~\(C(X\times_{\tau}G)\) is \(C(G)\)-equivariant, and inspection
of the formula~\eqref{eq:diag-twisted-tensor-explicit} shows that so is the diagonal on~\(C(X)\otimes_{t}C(G)\).
Because \(\psi=\psi_{G}\) is also \(C(G)\)-equivariant, we may assume \(y=1\in C(G)\). In other words,
it suffices to consider elements of the form~\(x\otimes 1\in C(X)\otimes_{t}C(G)\) when checking the claimed identity
\begin{equation}
  \Delta\,\psi = (\psi\otimes\psi)\,\Delta.
\end{equation}

We therefore need to look at \(\Delta\,\psi(x)=(-1)^{\deg{\ii}}\,\Delta\,\hatSz_{\ii}x\).
Combining \Cref{thm:Sz-Psi} with \Cref{thm:Psi-pp-degree}, we have
\begin{equation}
  \label{eq:hatSz-partial0}
  (\partial_{0})^{l}\hatSz_{\ii} x = \hatSz_{\ii_{2}} x(p_{1}-1,\dots,p_{k+1}-1)
\end{equation}
and
\begin{multline}
  \label{eq:hatSz-tpartial}
  \sum_{\ii_{1}\in S_{n}(\pp)}(-1)^{\deg{\ii}}\,\tpartial^{k}\,\hatSz_{\ii} x = \sum \,(-1)^{\epsilon}\,\hatSz_{\jj_{1}} x(0,\dots,p_{1}-1) \\*
  \cdot \Sz_{\jj_{2}} x(p_{1}-1,\dots,p_{2}-1) \cdots \Sz_{\jj_{k+1}} x(p_{k}-1,\dots,p_{k+1}-1),
\end{multline}
where the sum on the right-hand side is over all~\(\jj_{1}\in S_{q_{1}-1}\),~\dots,~\(\jj_{k+1}\in S_{q_{k+1}-1}\),
and
\begin{equation}
  \label{eq:hatSz-epsilon}
  \epsilon = \deg{\jj_{1}} + \dots + \deg{\jj_{k}} + \sum_{s=1}^{k}(s-1)(q_{s}-1).
\end{equation}

Also, formula~\eqref{eq:diag-twisted-tensor-explicit} for the diagonal on~\(C(X)\otimes_{t}C(G)\) boils for~\(x\otimes1\) down to
\begin{multline}
  \label{eq:diag-x-1}
  \quad
  \Delta\,(x\otimes 1) =
  \sum_{i=0}^{n}\sum_{k=0}^{n-i}\sum_{\pp}\,
  (-1)^{\epsilon(\pp)+i-\deg{z^{\pp}_{k+1}}} \\*
  \cdot \Bigl(\tpartial^{i}\,x \otimes t(z^{\pp}_{1})\cdots t(z^{\pp}_{k}) \Bigr)
  \otimes \Bigl( z^{\pp}_{k+1}\otimes 1 \Bigr)
  \quad
\end{multline}
where \(x\in X_{n}\), \(z=(\partial_{0})^{n-i}\,x\in X_{i}\),
and the last sum is over all interval cuts~\(\pp\) of~\(\setzero{i}\) corresponding to~\(e_{k}\).
To this expression we have to apply the map~\(\psi\otimes\psi\).
Note that the first tensor factor above is of the form
\begin{multline}
  \label{eq:diag-x-1-first}
  \tpartial^{i}\,x \otimes t(z^{\pp}_{1})\cdots t(z^{\pp}_{k})
  = \tpartial^{i}\,x \otimes \Sz z^{\pp}_{1}\cdots \Sz z^{\pp}_{k} \\
    {} + \text{additional terms with fewer than~\(k\) factors in the second component.}
\end{multline}
As in \Cref{sec:proof-main}, these additional terms arise whenever a~\(z^{\pp}_{s}\) with~\(1\le s\le k\)
is of degree~\(1\) because of the extra term~\(-1\in C(G)\) in the definition of~\(t\) in this case.

We first consider the cuts~\(\pp\) in~\eqref{eq:diag-x-1}
with~\(\ell(\pp)=0\), that is, where the intervals with labels~\(1\) to~\(k\) cover all of~\([i]\).
In this case we conclude the following from~\eqref{eq:hatSz-partial0} and~\eqref{eq:hatSz-tpartial}:
If we apply \(\psi\otimes\psi\) to the terms in~\eqref{eq:diag-x-1}
that correspond to the first line of~\eqref{eq:diag-x-1-first}, then we exactly get the terms appearing in
\begin{equation}
  \sum_{\ii_{1}\in S_{n}(\pp)} (-1)^{\deg{\ii}}\, \tpartial^{k}\,\hatSz_{\ii} x \otimes (\partial_{0})^{l}\,\hatSz_{\ii} x
\end{equation}
if we set \(i=p_{1}-1\) and~\(z=x(n-i,\dots,n)\). Moreover, the formula~\eqref{eq:hatSz-epsilon} tells us
that the sign above corresponds with the one in~\eqref{eq:diag-x-1}.

We now proceed to showing that the decompositions~\(\pp\) with~\(\ell(\pp)>0\) lead to summands in~\eqref{eq:diag-x-1}
that cancel out with the additional terms in~\eqref{eq:diag-x-1-first} for the~\(\pp\) with~\(\ell(\pp)=0\).
The variable~\(i\in[n]\) in~\eqref{eq:diag-x-1} is fixed during the following discussion.

We look at a minimal decomposition~\(\pp\) of~\([i]\) according to the partial ordering introduced in~\Cref{sec:proof-main}
and at the \(2^{\ell_{1}(\pp)}\)~decompositions~\(\pp'\ge\pp\). They all lead to the same~\(z^{\pp'}_{k'+1}=z^{\pp}_{k+1}\),
hence to the same second tensor factor~\(\hatSz z^{\pp}_{k+1}\) in
\begin{equation}
  (\psi\otimes\psi)\,\Delta\,(x\otimes1).
\end{equation}
For each such~\(\pp'\), the first tensor factor in~\eqref{eq:diag-x-1},
\begin{equation}
  (-1)^{\epsilon(\pp')+i+\deg{z^{\pp'}_{k'+1}}}\,\tpartial^{i}\,x \otimes t(z^{\pp'}_{1})\cdots t(z^{\pp'}_{k'}),
\end{equation}
contains the term
\begin{equation}
  (-1)^{+\ell_{1}(\pp')}\Bigl( (-1)^{\epsilon(\pp)+i+\deg{z^{\pp}_{k+1}}}\,\tpartial^{i}\,x \otimes \Sz z^{\pp}_{1}\cdots \Sz z^{\pp}_{k} \Bigr), 
\end{equation}
because of the contributions~\(-1\in C(G)\) of each interval of length~\(1\),
and also because we have \(\epsilon(\pp')=\epsilon(\pp)\) by \Cref{thm:pp-refine}.
As before, these terms add up to~\(0\) for~\(\ell_{1}(\pp)>0\), which completes the proof.

\section{Comparison with Shih's twisted tensor product}
\label{sec:shih}

We have mentioned in the introduction already that Szczarba's twisting cochain agrees
with the one constructed by Shih~\cite[\S II.1]{Shih:1962} using homological perturbation theory.
In~\cite[Sec.~7]{Franz:szczarba1} we pointed out that despite this agreement
their approaches lead to different twisted tensor products and different twisted shuffle maps.

Recall that given any cochain~\(t\colon C\to A\), one can define the twisted tensor products
\begin{equation}
  \label{eq:both-twisted-tensor-products}
  C\otimes_{t}M
  \qquad\text{and}\qquad
  M\otimes_{t}C
\end{equation}
for a left or, respectively, right \(A\)-module, see~\cite[Def.~II.1.4]{HusemollerMooreStasheff:1974} for instance.
The twisted tensor products considered so far have been of the first kind.

In \Cref{sec:morph-dgc} we have proven that Szczarba's twisted shuffle map
\begin{equation}
  \psi\colon C(X)\otimes_{t}C(F) \to C(X\times_{\tau}F)
\end{equation}
is a morphism of dgcs, and and it is not difficult to see
that for~\(F=G\) it is also a morphism of right \(C(G)\)-modules \cite[Prop.~7.1]{Franz:szczarba1}.

Shih on the other hand uses the twisted tensor product~\(C(F)\otimes_{t}C(X)\) (where the fibre~\(F\)
is considered as a right \(G\)-space). His twisted shuffle map
\begin{equation}
  \shuffle^{\tau}\colon C(F)\otimes_{t}C(X) \to C(F\times_{\tau}X)
\end{equation}
is part of a contraction that is a homotopy equivalence of right \(C(X)\)-comodules and, in the case~\(F=G\), of left \(C(G)\)-modules,
see~\cite[Props.~II.4.2~\&~II.4.3]{Shih:1962} and~\cite[Lemma~4.5\(^{*}\)]{Gugenheim:1972}. In this sense his result is stronger because
it is not known whether Szczarba's map~\(\psi\) is part of such a homotopy equivalence.\footnote{%
  Since the underlying complexes are free and defined for~\(\kk=\Z\),
  the map~\(\psi\) is at least a homotopy equivalence of complexes, \cf~\cite[Prop.~II.4.3]{Dold:1980}.}

On the other hand, there does not seem to be a dgc structure on~\(C(F)\otimes_{t}C(X)\).
The ``mirror image'' of~\eqref{eq:def-diag-twisted} gives a chain map
\begin{equation}
  \label{eq:def-diag-mirror}
  C(F)\otimes_{t}C(X) \to \Bigl( C(F)\otimes_{t}C(X) \Bigr) \otimes \Bigl( C(F)\otimes_{t}C(X) \Bigr),
\end{equation}
but it is not coassociative in general because of the asymmetry
inherent in the definition of the cooperations~\(E^{k}\).
We expect, however, that \eqref{eq:def-diag-mirror} extends to an \(A_{\infty}\)-coalgebra structure.

There is a different definition of an hgc, based on cooperations
\begin{equation}
  \tilde E^{k}\colon C \to C \otimes C^{\otimes k},
\end{equation}
which for simplicial sets is realized by the interval cut operations~\(\tilde E^{k}=\AWu{\tilde e_{k}}\)
based on the surjections~\(\tilde e_{k}=(1,2,1,\dots,1,k,1)\), \cf~\cite[Sec.~4]{Franz:gersten}.
In this setting \(C(F)\otimes_{t}C(X)\) would become a dgc with the diagonal~\eqref{eq:def-diag-mirror}
if Szczarba's twisting cochain~\(t\) were comultiplicative with respect to this new hgc structure.
This is not the case, however, as can be seen for~\(\CobarEl{x}\in\OM\,C(X)\) with~\(x\in X_{2}\) already.

\section{Discrete fibres}
\label{sec:cover}

In this section we dualize the dgc model from \Cref{thm:psi-dgc}
to a dga model for bundles with finite fibres.
We also derive a certain spectral sequence converging to the homology of a bundle with discrete fibre
that in the context of CW~complexes was constructed by Papadima--Suciu~\cite{PapadimaSuciu:2010}.
For finite fibres we again consider the dual spectral sequence converging to the cohomology of the bundle,
which turns out to be a spectral sequence of algebras.
In the special case of a \(p\)-group it has recently been studied by Rüping--Stephan~\cite{RuepingStephan:2019}.

\subsection{The homological spectral sequence}

Let \(G\) be (the simplicial group associated to) a discrete group, so that \(C(G)=C_{0}(G)=\kk[G]\) is the group ring
with coefficients in~\(\kk\). We write \(\aa\lhd\kk[G]\) for the augmentation ideal.
For a discrete space~\(F\) it gives rise to an increasing filtration of~\(C(F)=C_{0}(F)\) by the \(\kk[G]\)-submodules
\begin{equation}
  \FF_{-p}(F) = \aa^{p}\,C(F)
\end{equation}
with~\(p\in\N\) (and the convention~\(\aa^{0}=\kk[G])\). We write \(\gr_{*}(F)\) for the associated graded module
over the graded algebra~\(\gr_{*}(G)\) with structure map~\(\gr_{*}\mu\) induced by the action~\(\mu\colon G\times F\to F\).

Given a bundle~\(X\times_{\tau}F\), we consider the increasing filtration
\begin{equation}
  \label{eq:def-FFp}
  \FF_{-p}(X,F) = C(X)\otimes_{t}\FF_{-p}(F)
\end{equation}
of the twisted tensor product~\(C(X)\otimes_{t}C(F)\) by subcomplexes.
The zeroeth page of the associated spectral sequence is of the form
\begin{equation}
  \Ess^{0}_{p,q} = C_{p}(X) \otimes \gr_{q}(F)
\end{equation}
and lives in the lower half-plane as \(q\le0\).

Since \(G\) is discrete, any twisting cochain mapping to~\(C(G)\) vanishes in all degrees different from~\(1\).
It furthermore takes values in the augmentation ideal~\(\aa\)
by the second defining identity in~\eqref{eq:def-tw}.
Hence the twisting term in the differential of~\(C(X)\otimes_{t}C(F)\) lowers the filtration degree.
As a result, the induced differential on~\(\Ess^{0}\) is \(d^{0}=d\otimes 1\),
and the first page of the spectral sequence is of the form
\begin{equation}
  \Ess^{1}_{p,q} = H_{p+q}(X;\gr_{q}(F)).
\end{equation}
The convergence of this spectral sequence is delicate in general, see~\cite[Sec.~5.3]{PapadimaSuciu:2010}.
However, if the augmentation ideal~\(\aa\) is nilpotent, meaning that \(\aa^{L}=0\) for some~\(L\),
then the filtration is finite and convergence is not an issue.

Let us assume that that \(\kk\) is a field or, more generally,
that \(H(X)\) is torsion-free over the principal ideal domain~\(\kk\).
We then have
\begin{equation}
  \Ess^{1}_{p,q} = H_{p+q}(X) \otimes \gr_{q}(F).
\end{equation}
Moreover, \(H(X)\) is a graded coalgebra in this case via the composition
\begin{equation}
  \label{eq:HX-coalg}
  H(X) \longrightarrow H(X\times X) \stackrel{\cong}{\longrightarrow} H(X)\otimes H(X)
\end{equation}
where the second map is the inverse of the Künneth isomorphism.

We need the following observation.

\begin{lemma}
  Let \(C\) be a dgc and \(G\) a discrete group, and let \(t\colon C\to \kk[G]\) be a twisting cochain.
  Then \(t\) induces a well-defined twisting cochain
  \begin{equation*}
    t_{*}\colon H(C)\to \gr_{*}(G),
    \qquad
    [c] \mapsto
    \begin{cases}
      [t(c)] \in \gr_{-1}(G) & \text{if \(\deg{c}=1\),} \\
      0 & \text{otherwise.}
    \end{cases}
  \end{equation*}
\end{lemma}

\begin{proof}
  For well-definedness we have to show \(t(d\,c)\in\FF_{-2}(G)\) for~\(c\in C_{1}\).
  Since there is no differential on~\(\kk[G]\),
  we get from the twisting cochain condition~\eqref{eq:def-tw} that
  \begin{equation}
    t(d\,c) = d\,t(c) + t(d\,c) = (t\cup t)(c) \in \FF_{-2}(G),
  \end{equation}
  again because \(t\) takes values in the augmentation ideal~\(\aa=\FF_{-1}(G)\).

  For degree reasons this also shows that \(t_{*}\) is a twisting cochain.
\end{proof}

The differential on the first page of the spectral sequence is given by the twisting term~\eqref{eq:def-delta-t}.
Using the lemma above and the fact that \(H(X)\) is a coalgebra, we can see that this differential
is the composition
\begin{multline}
  \qquad
  \delta_{t_{*}}\colon H(X)\otimes \gr_{*}(F) \xrightarrow{\Delta\otimes 1} H(X)\otimes H(X)\otimes \gr_{*}(F) \\*
  \xrightarrow{1\otimes t_{*}\otimes 1} H(X)\otimes \gr_{*}(G) \otimes \gr_{*}(F)
  \xrightarrow{1\otimes\gr_{*}\mu} H(X)\otimes \gr_{*}(F).
  \qquad
\end{multline}
In other words, we have an isomorphism of complexes
\begin{equation}
  \Ess^{1} = H(X) \otimes_{t_{*}} \gr_{*}(F).
\end{equation}
We thus recover the description of the spectral sequence of an equivariant chain complex
as given by Papadima--Suciu~\cite[Thm.~A]{PapadimaSuciu:2010}, up to the order of the tensor factors.

We now look at coalgebra structures.
The filtration~\(\FF(G)\) is comultiplicative in the sense that we have
\begin{equation}
  \Delta\,\FF_{-p}(F) \subset \sum_{q+r=p}\FF_{-q}(F)\otimes\FF_{-r}(F)
\end{equation}
for all~\(p\). In fact, the claim holds for the bialgebra~\(F=G\) by induction,
starting with the case~\(p=1\), which says that
\begin{align}
  \Delta(g-1) &= g\otimes g - 1\otimes 1 = g\otimes(g-1) + (g-1)\otimes1 \\
  \notag &\in \FF_{0}(G)\otimes\FF_{-1}(G) + \FF_{-1}(G)\otimes\FF_{0}(G)
\end{align}
for any~\(g\in G\). It carries over to~\(F\) as \(C(F)\) is a \(\kk[G]\)-dgc.

Moreover, inspection of formula~\eqref{eq:def-diag-twisted} or~\eqref{eq:diag-twisted-tensor-explicit}
for the diagonal of~\(C(X)\otimes_{t}C(F)\) shows that the filtration~\(\FF(X,G)\) is comultiplicative, too.
Taking again into account that the twisting cochain~\(t\) takes values in the augmentation ideal~\(\aa\),
we see that the diagonal on the page~\(\Ess^{0}\)
of the spectral sequence is componentwise,
\begin{equation}
  \Delta(c\otimes m) = \sum_{(c),(m)} \bigl(c_{(1)}\otimes m_{(1)}\bigr) \otimes \bigl(c_{(2)}\otimes m_{(2)}\bigr)
\end{equation}
for~\(c\in C_{p}(X)\) and~\(m\in\FF_{-q}(F)/\FF_{-q-1}(F)\).
This implies the following.

\begin{proposition}
  Assume that \(\kk\) is a field.
  The filtration~\eqref{eq:def-FFp} gives rise to a spectral sequence of coalgebras.
  As a dgc, its first page is given by
  \begin{equation*}
    \Ess^{1}_{p,q} = H_{p+q}(X) \otimes_{t_{*}} \gr_{q}(F)
  \end{equation*}
  with the componentwise coproduct. If the augmentation ideal~\(\aa\) is nilpotent,
  then the spectral sequence converges to~\(H(X\times_{\tau}F)\) as a graded coalgebra.
\end{proposition}

\subsection{Dga models and the cohomological spectral sequence}

We now turn to cohomology.
For the following purely algebraic reason
we restrict to finite structure groups~\(G\) and finite fibres~\(F\).

The dual~\(C^{*}\) of a dgc~\(C\) with coproduct~\(\Delta\) is a dga with the transpose~\(\Delta^{*}\)
as multiplication, or more precisely, with the composition
\begin{equation}
  C^{*}\otimes C^{*} \to (C\otimes C)^{*} \xrightarrow{\Delta^{*}} C^{*}.
\end{equation}
However, the dual of a dga~\(A\) is not a dgc in general, but it is so if C is finitely generated
free \(\kk\)-module in each degree. The coproduct is the transpose~\(\mu^{*}\) of the multiplication
or rather its composition with the isomorphism~\((A\otimes A)^{*}\cong A^{*}\otimes A^{*}\).

So let us assume that \(G\) is finite.\footnote{%
  This restriction is missing for the multiplicative model stated in~\cite[p.~219]{KadeishviliSaneblidze:2005}.
  Together with the assumption of simple connectedness made there for the base space~\(X\) (see \Cref{fn:kadeishvili-saneblidze-pi1}),
  that model boils down to the tensor product~\(C^{*}(X)\otimes C^{*}(F)\) for Cartesian products
  satisfying an appropriate finiteness condition.}
Then \(C^{*}(G)\) is a dgc, and of course \(C^{*}(X)\) is a dga for any~\(X\).
Because of the definition
\begin{equation}
  \label{eq:diff-dual}
  d_{C^{*}}=-d_{C}^{*}
\end{equation}
of the differential on a dual complex
as the \emph{negative} of the transpose of the original one
(compare~\cite[Sec.~2.3]{Franz:gersten}), the transpose
\begin{equation}
  t^{*}\colon C^{*}(G) \to C^{*}(X)
\end{equation}
of Szczarba's twisting cochain satisfies
\begin{align}
  d(t^{*}) &= d_{C^{*}(X)}\,t^{*} + t^{*}\,d_{C^{*}(G)}
  = - \bigl( d_{C(X)}^{*}\,t^{*} + t^{*}\,d_{C(G)}^{*} \bigr) \\*
  \notag &= - \bigl( t\,d_{C(X)} + d_{C(G)}\,t \bigr)^{*}
  = - d(t)^{*} = - (t\cup t)^{*} = - t^{*} \cup t^{*}.
\end{align}
In other words, \(u=-t^{*}\) is again a twisting cochain in our sense.

The quasi-isomorphism~\(C(X)\otimes_{t}C(F)\to C(X\times_{\tau}F)\) from \Cref{thm:psi-dgc}
dualizes to a quasi-isomorphism of dgas between~\(C^{*}(X\times_{\tau}F)\) and the dual of~\(C(X)\otimes_{t}C(F)\).
If the fibre~\(F\) is finite, then we have an isomorphism of complexes
\begin{equation}
  \label{eq:twisted-dga}
  \bigl( C(X)\otimes_{t} C(F) \bigr)^{*} = C^{*}(X) \otimes_{u} C^{*}(F),
\end{equation}
which is now a twisted tensor product of the second form in~\eqref{eq:both-twisted-tensor-products}.
The minus sign in~\(u=-t^{*}\) arises again from~\eqref{eq:diff-dual}
and also reflects the sign difference between the two kinds of
twisted tensor products, see again~\cite[Def.~II.1.4]{HusemollerMooreStasheff:1974}.

The product on~\eqref{eq:twisted-dga} is as described by 
Kadeishvili--Saneblidze~\cite[eq.~(12)]{KadeishviliSaneblidze:2005}.
With our sign convention and in Sweedler notation it is of the form
\begin{equation}
  \label{eq:def-twisted-mult}
  (a\otimes b)\cdot(a'\otimes b')
  = \sum_{k\ge0} \sum_{(b)} \, (-1)^{k} \, a\,E_{k}(u(\swee{b}{1}),\dots,u(\swee{b}{k});a')\otimes \swee{b}{k+1}\,b'
\end{equation}
for~\(a\),~\(a'\in C^{*}(X)\) and~\(b\),~\(b'\in C^{*}(F)\). The transposes
\begin{equation}
  E_{k} = \bigl( E^{k} \bigr)^{*} \colon C^{*}(X)^{\otimes k}\otimes C^{*}(X) \to C^{*}(X)
\end{equation}
are the structure maps of the hga~\(C^{*}(X)\), see \Cref{rem:hga}.
Note that the sum over~\(k\) in~\eqref{eq:def-twisted-mult} is in fact only over~\(0\le k\le\deg{b}+\deg{a'}\)
because of the vanishing condition~\eqref{eq:Ek-0}.

We summarize our discussion so far as follows.

\begin{proposition}
  Let \(X\times_{\tau}F\) be a fibre bundle where both the fibre~\(F\) and the structure group~\(G\) have only finitely many non-degenerate simplices.
  Then the dga~\(C^{*}(X)\otimes_{u}C^{*}(F)\) with the product~\eqref{eq:def-twisted-mult} is a model for~\(X\times_{\tau}F\).
  The quasi-iso\-mor\-phism connecting this dga with~\(C^{*}(X\times_{\tau}F)\) is natural in~\(X\),~\(G\) and~\(F\).
\end{proposition}

We now look at the duals of the filtrations introduced in the previous section.
Because the filtrations~\(\FF(F)\) and~\(\FF(X,F)\) are comultiplicative,
the dual filtrations of~\(C^{*}(F)=C^{0}(F)\) and~\(C^{*}(X)\otimes_{u}C^{*}(F)\),
\begin{align}
  \label{eq:def-FFp-dual}
  \FF^{-p}(F) &= \bigl\{\, \gamma\in C^{0}(F) \bigm| \text{\(\gamma(m)=0\) for all~\(m\in\FF_{-p-1}(F)\)} \,\bigr\}, \\
  \FF^{-p}(X,F) &= C^{*}(X) \otimes_{u} \FF^{-p}(F)
\end{align}
are multiplicative. Specializing to field coefficients, we arrive at the following conclusion.
It generalizes a result of Rüping--Stephan~\cite[Cor.~4.19]{RuepingStephan:2019} for finite \(p\)-groups and
coefficients of prime characteristic~\(p\), see also~\cite[Rem.~4.20]{RuepingStephan:2019}.

\begin{proposition}
  Let \(\kk\) be a field, and let \(G\) be a finite group such that the augmentation ideal~\(\aa\lhd\kk[G]\) is nilpotent.
  There is a multiplicative spectral sequence~\(\Ess_{r}\) converging to~\(H^{*}(X\times_{\tau}F)\)
  whose first page is of the form
  \begin{equation*}
    \Ess_{1}^{p,q} = H^{p+q}(X) \otimes \gr^{q}(F)
  \end{equation*}
  with componentwise product, where \(\gr^{*}(F)\) is the graded algebra associated to the filtration~\eqref{eq:def-FFp-dual}.
  The spectral sequence is natural in~\(X\), \(G\) and~\(F\).
\end{proposition}

\section{The Serre spectral sequence}
\label{sec:serre}

\Cref{thm:psi-dgc} allows for a short proof
of the product structure in the cohomological Serre spectral sequence.
The same applies to the comultiplicative structure in the homological setting
considered by Chan~\cite[Thm.~1.2]{Chan:serre}.
We assume throughout this section that \(\kk\) is a principal ideal domain.

Recall that if the homology~\(H(C)\) of a dgc~\(C\) is free over~\(\kk\),
then it is a graded coalgebra with diagonal
\begin{equation}
  H(C) \longrightarrow H(C \otimes C) \stackrel{\cong}{\longrightarrow} H(C) \otimes H(C)
\end{equation}
where the last map is the inverse of the Künneth isomorphism.
(We have mentioned a special case of this already in~\eqref{eq:HX-coalg}.)

\begin{proposition}
  Let \(E=X\times_{\tau}F\) be a twisted Cartesian product with the simplicial group~\(G\) as structure group.
  \begin{enumroman}
  \item Assume that \(H(X)\) and~\(H(F)\) are free over~\(\kk\) and that \(G_{0}\) acts trivially on~\(H(F)\).
    The homological Serre spectral sequence is a spectral sequence of coalgebras with the componentwise coproduct on
    \begin{equation*}
      \Ess^{2}_{pq} = H_{p}(X) \otimes H_{q}(F),
    \end{equation*}
    converging to~\(H(E)\) as a coalgebra.
  \item Assume that \(F\) is of finite type, that \(H^{*}(X)\) or~\(H^{*}(F)\) is flat over~\(\kk\) and that \(G_{0}\) acts trivially on~\(H^{*}(F)\).
    The cohomological Serre spectral sequence is a spectral sequence of algebras with the componentwise product on
    \begin{equation*}
      \Ess_{2}^{pq} = H^{p}(X) \otimes H^{q}(F),
    \end{equation*}
     converging to~\(H^{*}(E)\) as an algebra.
  \end{enumroman}
\end{proposition}

\begin{proof}
  By \Cref{thm:psi-dgc}, the dgc~\(C(E)\) is quasi-isomorphic
  to~\(M=C(X)\otimes_{t}C(F)\) with the coproduct~\eqref{eq:def-diag-cobar-CX}.
  We filter~\(M\) by increasing degree in~\(C(X)\) and then \(M\otimes M\) via the tensor product filtration.
  Let \(\Ess^{r}\) be the associated spectral sequence converging to~\(H(M)\)
  and \(\Fss^{r}\) the one converging to~\(H(M\otimes M)\).
  
  Since \(G_{0}\) acts trivially on~\(H(F)\), the definition~\eqref{eq:def-tsz} of Szczarba's twisting cochain
  tells us that this module is annihilated by~\(t(x)\) for any~\(x\in X_{1}\).
  Therefore,
  \begin{align}
    \Ess^{0}_{pq} &= C_{p}(X) \otimes C_{q}(F), & d^{0} &= 1 \otimes d, \\
    \Ess^{1}_{pq} &= C_{p}(X) \otimes H_{q}(F), & d^{1} &= d \otimes 1,  \\
    \Ess^{2}_{pq} &= H_{p}(X) \otimes H_{q}(F)
  \end{align}
  and similarly
  \begin{align}
    \Fss^{1}_{pq} &= \bigoplus_{p_{1}+p_{2}=p}\bigoplus_{q_{1}+q_{2}=q} \,
    C_{p_{1}}(X) \otimes H_{q_{1}}(F) \otimes C_{p_{2}}(X) \otimes H_{q_{2}}(F), \\
    \Fss^{2}_{pq} &= \bigoplus_{p_{1}+p_{2}=p}\bigoplus_{q_{1}+q_{2}=q} \,
    H_{p_{1}}(X) \otimes H_{q_{1}}(F) \otimes H_{p_{2}}(X) \otimes H_{q_{2}}(F).
  \end{align}
  
  Inspection of the formula~\eqref{eq:def-diag-cobar-CX} shows that the coproduct is filtration-pre\-serv\-ing
  and that the induced maps between the first and second pages of the spectral sequences are the componentwise diagonals:
  In the notation of \Cref{sec:proof-main,sec:morph-dgc}, summands corresponding to partitions~\(\pp\)
  with~\(\ell_{1}(\pp)>0\) do not contribute, again by the annihilation property of~\(t\) mentioned above,
  and among the remaining ones those with~\(\ell(\pp)<p\)
  end up in a lower filtration degree. This proves the first part.

  The transpose~\(\psi^{*}\colon C^{*}(E) \to M^{*}\)
  of~\(\psi\) is a quasi-isomorphism of dgas. We filter \(M^{*}\) by the dual filtration,
  which leads to a spectral sequence~\(\Ess_{r}\) converging to~\(H^{*}(E)\).
  Since \(F\) is of finite type, we have
  \begin{align}
    \Ess_{0}^{pq} &= \bigl( C_{p}(X) \otimes C_{q}(F) \bigr)^{*}, \\
    \label{eq:E1-cohom}
    \Ess_{1}^{pq} &= C^{p}(X) \otimes H^{q}(F) \\
    \shortintertext{by the cohomological Künneth theorem~\cite[Prop.~VI.10.24, case~II]{Dold:1980}, hence}
    \label{eq:E2-cohom}
    \Ess_{2}^{pq} &= H^{p}(X) \otimes H^{q}(F)
  \end{align}
  by its homological counterpart~\cite[Thm.~VI.9.13]{Dold:1980} and the assumption that \(t(x)\) annihilates \(H^{*}(F)\) for any~\(x\in X_{1}\).
  By the same argument as before, the products on~\eqref{eq:E1-cohom} and~\eqref{eq:E2-cohom} are component\-wise.
  This concludes the proof.
\end{proof}

\appendix

\section{Comparison with Baues' diagonal}
\label{sec:Baues}

Baues \cite[Sec.~1]{Baues:1981} has defined a diagonal on~\(\OM\,C(X)\)
for any \(1\)-reduced simplicial set~\(X\). In this appendix we compare his map
with the diagonal~\eqref{eq:def-diag-cobar-CX} induced by the hgc structure of~\(C(X)\)
(which of course is defined for any~\(X\ne\emptyset\)).
Up to sign, this has already been done by Quesney~\cite[Prop.~5.1]{Quesney:2016}.

\begin{proposition}
  \label{thm:Baues-Ek}
  For a \(1\)-reduced simplicial set~\(X\) the diagonal~\eqref{eq:def-diag-cobar-CX} on~\(\OM\,C(X)\) is the same as Baues'.
\end{proposition}

This implies that the diagonal~\eqref{eq:def-diag-cobar-CX} is also equal
to the one constructed by Hess--Parent--Scott--Tonks
via homological perturbation theory \cite[Secs.~4~\&~5]{HessEtAl:2006}.\footnote{%
  Hess--Parent--Scott--Tonks state that their recursively defined diagonal agrees with Baues'.
  This includes the sign of each summand~\eqref{eq:Baues-term},
  which is not made explicit for their own formula.} 

\begin{proof}
  Let \(x\in X\) be an \(n\)-simplex.
  The terms in Baues' formula for~\(\Delta\,\CobarEl{x}\) \cite[p.~334]{Baues:1981}
  are indexed by the subsets~\(\JJ\subset\setone{n-1}\).
  It not difficult to see
  that in analogy with formula~\eqref{eq:def-EE} Baues' diagonal is of the form
  \begin{equation}
    \Delta\,\CobarEl{x} = \CobarEl{x} \otimes 1 + \sum_{k=0}^{\infty}\EEt^{k}(x)
  \end{equation}
  for certain functions
  \begin{equation}
    \EEt^{k}\colon C(X) \to \OM_{k}\,C(X) \otimes \OM_{1}\,C(X).
  \end{equation}
  Moreover, each non-zero summand appearing in~\(\EEt^{k}(x)\) can be written as
  \begin{equation}
    \label{eq:Baues-term}
    \pm \bigCobarEl{ x^{\pp}_{1} \bigm| \dots \bigm| x^{\pp}_{k} } \otimes \CobarEl{ x^{\pp}_{k+1} }
  \end{equation}
  for the unique interval cut~\(\pp\) of~\(\setzero{n}\) associated to~\(e_{k}\) such that
  \(x^{\pp}_{k+1}\) contains the vertices indexed by~\(\JJ\) plus~\(0\) and~\(n\).
  Hence, up to sign, we get the claimed identity
  \begin{equation}
    \label{eq:Baues-new}
    \susp^{\otimes(k+1)}\,\EEt^{k}(x) = \susp^{\otimes(k+1)}\,\EE^{k}(x) = (-1)^{k}\AWu{e_{k}}(x).
  \end{equation}
  
  It remains to verify the sign, where we proceed by induction on~\(k\).
  The case~\(k=0\) is trivial because \(\AW_{(1)}\) is the identity map
  and \(\EEt^{0}=\desusp\) the inverse of the suspension map~\(\susp\).

  For~\(k>1\) we compare the signs associated to an interval cut
  \begin{equation}
    \pp\, \colon p_{0} \rr{k+1} p_{1} \rr{1} p_{2} \rr{k+1}
    \cdots
    \rr{k+1} p_{2k-1} \rr{k} p_{2k} \rr{k+1} p_{2k+1}
  \end{equation}
  for the surjection~\(e_{k}\) with those for the interval cut
  \begin{equation}
    \pp' \colon p_{0} \rr{k} p_{1} \rr{1} p_{2} \rr{k}
    \cdots
    \rr{k} p_{2k-1} \phantom{ \rr{k} p_{2k} \rr{k} p_{2k+1} }
  \end{equation}
  for~\(e_{k-1}\). We compute the exponents of all the signs involved,
  always modulo~\(2\). The exponents of the permutation signs differ by
  \begin{align}
    \perm(\pp)-\perm(\pp') &\equiv  
    (p_{2k}-p_{2k-1})\Biggl(p_{1}+1+\sum_{i=1}^{k-1}(p_{2i+1}-p_{2i}+1)\Biggr) \\
    \notag &\equiv (p_{2k}-p_{2k-1})\Biggl(\,\sum_{i=1}^{2k-1}p_{i}+k\Biggr)
  \end{align}
  \def\uu{e_{k}}
  since we have to move the interval corresponding to~\(\uu(2k)=k\)
  before all preceding (inner) intervals
  corresponding to~\(\uu(1)=\uu(3)=\dots=\uu(2k-1)=k+1\).
  The exponents of the position signs change by~\(p_{2k-1}\)
  because of the additional inner interval for~\(\uu(2k-1)=k+1\).
  
  The sign for the summand~\eqref{eq:Baues-term} is the sign
  of the shuffle\footnote{%
    Strictly speaking, this is not a shuffle in the sense of \Cref{sec:shuffle} as \(0\notin\setone{n-1}=\{1,\dots,n-1\}\).}%
  ~\((\setone{n-1}\setminus \JJ, \JJ)\).
  Hence, by passing from~\(k-1\) to~\(k\), the exponent of this sign changes by
  \begin{multline}
    \qquad
    (p_{2k}-p_{2k-1}-1)\Biggl(p_{1}+\sum_{i=1}^{k-1}(p_{2i+1}-p_{2i}+1)\Biggr) \\
    \equiv (p_{2k}-p_{2k-1}-1)\Biggl(\,\sum_{i=1}^{2k-1}p_{i}+k+1\Biggr)
    \qquad
  \end{multline}
  because we have to move all elements in the interior of the \(k\)-th interval
  before all previous values occurring in~\(\JJ\), that is, all vertices in~\(x^{\pp}_{k+1}\)
  with indices strictly between~\(0\) and~\(p_{2k-1}\).

  Still modulo~\(2\), the changes in the exponents add up to
  \begin{equation}
    \sum_{i=1}^{2k}p_{i}+k+1 \equiv \sum_{i=1}^{k}(p_{2i}-p_{2i-1}+1)+1
    \equiv \bigdeg{\, \bigCobarEl{ x^{\pp}_{1}|\dots|x^{\pp}_{k} } \,} + 1.
  \end{equation}
  This is exactly the exponent of the sign change
  we get when we pass from~\(k-1\) to~\(k\) in~\eqref{eq:Baues-new}.
  The sign exponent~\(\deg{\,\CobarEl{x^{\pp}_{1}|\dots|x^{\pp}_{k}}\,}\)
  arises because we have to move the additional suspension operator
  past the element~\(\CobarEl{x^{\pp}_{1}|\dots|x^{\pp}_{k}}\).
  Another minus sign comes from the increased exponent on the right-hand side of~\eqref{eq:Baues-new}.
  This completes the proof.
\end{proof}

\section{Szczarba operators and degeneracy maps}
\label{sec:szczarba-degen}

Apparently, neither in Szczarba's paper~\cite{Szczarba:1961} nor elsewhere in the literature one can find a proof
that Szczarba's twisting cochain~\eqref{eq:def-tsz} and his twisted shuffle map~\eqref{eq:def-psi}
are actually well-defined on normalized chain complexes. The purpose of this appendix is to close this gap.

Recall from~\cite[eq.~(3.1)]{Szczarba:1961} and~\cite[eq.~(6)]{HessTonks:2006}
that the simplicial operators
\begin{equation}
  \DD{k}{\ii}\colon X_{m} \to X_{m+k}
\end{equation}
for~\(0\le k\le n\),~\(\ii\in S_{n}\) and~\(m\ge n-k\) are recursively defined by
\begin{equation}
  \label{eq:def-DD}
  \DD{0}{\emptyset} = \id
  \qquad\text{and}\qquad
  \DD{k}{\ii} =
  \begin{cases}
    \DDd{k}{\ii'}\,s_{0}\,\partial_{i_{1}-k} & \text{if \(k<i_{1}\),} \\
    \DDd{k}{\ii'} & \text{if \(k=i_{1}\),} \\
    \DDd{k-1}{\ii'}\,s_{0} & \text{if \(k>i_{1}\)} \\
  \end{cases}
\end{equation}
for~\(n\ge1\) where \(\ii'=(i_{2},\dots,i_{n})\). Here \(D'\) denotes the derived operator
of a simplicial operator~\(D\), compare~\cite[p.~199]{Szczarba:1961} or~\cite[p.~1863]{HessTonks:2006}.

For~\(n\ge1\) we introduce a map
\begin{equation}
  \Phi\colon S_{n} \times [n] \to S_{n-1} \times [n-1],
  \qquad
  (\ii,p) \mapsto (\jj,q)
\end{equation}
recursively via
\begin{equation}
  \left\{\;
  \begin{alignedat}{4}
    \jj &= (i_{1}-1,\jj'), &\;\;\;  q &= q'+1 & \qquad\text{if\ \ } p &< i_{1}, &\quad\; (\jj',q') &\coloneqq \Phi(\ii',p), \\
    \jj &= \ii', &\;\;\;  q &= 0 & \qquad\text{if\ \ } p & = \rlap{\(i_{1}\)\ \ or\ \ \(i_{1}+1\),} \\
    \jj &= (i_{1},\jj'), &\;\;\;  q &= q'+1 & \qquad\text{if\ \ } p &> i_{1}+1, &\quad\; (\jj',q') &\coloneqq \Phi(\ii',p-1)
  \end{alignedat}
  \right.
\end{equation}
where again \(\ii'=(i_{2},\dots,i_{n})\).
Note that the base case~\(n=1\) is completely covered by the second line above since \(i_{1}=0\) in that case.

\begin{lemma}
  \label{thm:DD-s}
  Let~\(n\ge1\),~\(\ii\in S_{n}\) and~\(p\in[n]\), and set \((\jj,q)=\Phi(\ii,p)\).
  \begin{enumroman}
  \item For any~\(0\le k<p\) and any simplex~\(x\) of dimension~\(m\ge n-k-1\) we have
    \begin{equation*}
      \DD{k}{\ii}\,s_{p-1-k}\,x = s_{q}\,\DD{k}{\jj}\,x.
    \end{equation*}
  \item For any~\(p< k \le n\) and any simplex~\(x\) of dimension~\(m\ge n-k\) we have
    \begin{equation*}
      \DD{k}{\ii}\,x = s_{q}\,\DD{k-1}{\jj}\,x.
    \end{equation*}
  \end{enumroman}
\end{lemma}

\begin{proof}
  These are direct verifications by induction on~\(n\), based on the definitions of~\(\DD{k}{\ii}\) and~\(\Phi\).
  The base cases are \(\ii=(0)\), \(k=0\), \(p=1\) and \(\ii=(0)\), \(k=1\), \(p=0\), respectively.  
  In the induction step of the first formula, one distinguishes the cases~\(k<i_{1}\)
  (with the subcases~\(i_{1}<p-1\), \(i_{1}\in\{p-1,p\}\) and~\(i_{1}>p\)),
  \(k=i_{1}\) (with the subcases~\(i_{1}<p-1\) and~\(i_{1}=p-1\)) and~\(k>i_{1}\).
  For the second formula one has the cases~\(k<i_{1}\), \(k=i_{1}\) and~\(k>i_{1}\)
  (with the subcases~\(p<i_{1}\), \(p\in\{i_{1},i_{1}+1\}\) and~\(p>i_{1}+1\)).

  For instance, for~\(n>1\),~\(k<i_{1}\) and~\(i_{1}>p\) we have
  \begin{align}
    \Sz_{\ii}\,s_{p}\,x &= \DDd{k}{\ii'}\,s_{0}\,\partial_{i_{1}-k}\,s_{p-1-k}\,x
    = \DDd{k}{\ii'}\,s_{0}\,s_{p-1-k}\,\partial_{i_{1}-k-1}\,x \\
    \notag &= \DDd{k}{\ii'}\,s_{p-k}\,s_{0}\,\partial_{i_{1}-k-1}\,x
    = \bigl(\DD{k}{\ii'}\,s_{p-1-k}\bigr)' \, s_{0}\,\partial_{i_{1}-k-1}\,x \\
    \notag &= \bigl( s_{q'}\,\DD{k}{\jj'} \bigr)' \, s_{0}\,\partial_{i_{1}-k-1}\,x \\
    \shortintertext{by induction, where~\((\jj',q')=\Phi(\ii',p)\). Then \(\jj=(i_{1}-1,\jj')\) and~\(q=q'+1\), hence}
    \notag &=  s_{q'+1}\,\DDd{k}{\jj'} \, s_{0}\,\partial_{i_{1}-1-k}\,x
    = s_{q}\,\DDd{k}{\jj}\,x
  \end{align}
  since \(k<p\le i_{1}-1\).
\end{proof}

\begin{proposition}
  Let \(0\le p\le n\), and let \(x\) be an \(n\)-simplex.
  \begin{enumroman}
  \item For~\(\ii\in S_{n}\) and~\((\jj,q)=\Phi(\ii,p)\) we have
    \begin{equation*}
      \Sz_{\ii}\,s_{p}\,x = s_{q}\,\Sz_{\jj}\,x.
    \end{equation*}
  \item For~\(\ii\in S_{n+1}\) and~\((\jj,q)=\Phi(\ii,p+1)\) we have
    \begin{equation*}
      \hatSz_{\ii}\,s_{p}\,x = s_{q}\,\hatSz_{\jj}\,x.
    \end{equation*}
  \end{enumroman}
\end{proposition}

\begin{proof}
  These formulas follow from \Cref{thm:DD-s} and the identities~\eqref{eq:tau-sk} and~\eqref{eq:tau-s0}.
  For example, we have
  \begin{align}
    \Sz_{\ii}\,s_{p}\,x &=
    \DD{0}{\ii}\,\sigma(s_{p}\,x)\,\DD{1}{\ii}\,\sigma(\partial_{0}\,s_{p}\,x)\cdots \DD{n}{\ii}\,\sigma((\partial_{0})^{n}s_{p}\,x) \\
    \notag &= \DD{0}{\ii}\,s_{p-1}\,\sigma(x)\cdots\DD{p-1}{\ii}\,s_{0}\,\sigma((\partial_{0})^{p-1}\,x)
    \DD{p}{\ii}\,\sigma(s_{0}\,(\partial_{0})^{p}\,x) \\*
    \notag &\qquad\qquad\qquad\qquad\qquad\; {}\cdot \DD{p+1}{\ii}\,\sigma((\partial_{0})^{p}\,x)\cdots \DD{n}{\ii}\,\sigma((\partial_{0})^{n-1}\,x) \\
    \notag &= s_{q}\,\DD{0}{\jj}\,\sigma(x)\cdots s_{q}\,\DD{p-1}{\jj}\,\sigma((\partial_{0})^{p-1}\,x) \cdot 1 \\*
    \notag &\qquad\qquad\qquad\qquad {}\cdot s_{q}\,\DD{p}{\jj}\,\sigma((\partial_{0})^{p}\,x)\cdots s_{q}\,\DD{n-1}{\jj}\,\sigma((\partial_{0})^{n-1}\,x) \\
    \notag &= s_{q}\,\Sz_{\jj} x. \qedhere
  \end{align}
\end{proof}

\begin{corollary}
  Szczarba's twisting cochain~\(t\) and the twisted shuffle map~\(\psi\)
  descend to the normalized chain complexes.
\end{corollary}

\begin{proof}
  This is a consequence of the formulas just established and, for the twisting cochain~\(t\), the identity~\(t(s_{0}\,x)=\sigma(s_{0}\,x)-1=0\)
  for any \(0\)-simplex~\(x\).  
\end{proof}

\end{document}